\documentclass[12pt]{amsart}

\usepackage{amssymb,latexsym}
\usepackage{amsfonts}
\usepackage[mathscr]{eucal}
\usepackage{enumerate}
\usepackage[colorlinks,linkcolor=black,anchorcolor=black, citecolor=black]{hyperref}
\usepackage{fancyhdr}
\usepackage[deletedmarkup=xout,final]{changes}
\definechangesauthor[name={Per cusse}, color=orange]{per}

\newtheorem{theorem}{Theorem}[section]
\newtheorem{corollary}[theorem]{Corollary}
\newtheorem{remark}[theorem]{Remark}
\newtheorem{lemma}[theorem]{Lemma}

\newtheorem{proposition}[theorem]{Proposition}
\newtheorem{definition}{Definition}[section]

\numberwithin{equation}{section}

\def\ba{\begin{eqnarray}}
	\def\ea{\end{eqnarray}}
\def\tilde{\widetilde}

\def\e1{\epsilon}
\def\o1{\omega}
\def\01{\Omega}
\def\c1{\gamma}

\def\g1{\Sigma}

\def\l1{\Lambda}
\def\v1{\varphi}
\def\d1{\delta}
\def\part{\partial}

\def\f2{F}
\def\h2{{\bf H}}
\def\a2{{\bf A}}
\def\x2{{\bf X}}
\def\t1{\theta}
\def\b1{\beta}
\def\bar{\overline}
\def\bs{\begin{eqnarray*}}
	\def\es{\end{eqnarray*}}

\def\m1{\Theta}
\def\w1{\wedge}

\author{Xiaoli Han, Jiayu Li, Jun Sun}

\address{Xiaoli Han, Department of Mathematical Sciences, Tsinghua University \\ Beijing 100084, P. R. of China.}
\email{hanxiaoli@mail.tsinghua.edu.cn}

\address{Jiayu Li, School of Mathematical Sciences, University of Science and Technology of China Hefei 230026, P. R. of China.}
\email{jiayuli@ustc.edu.cn}

\address{Jun Sun, School of Mathematics and Statistics, Wuhan University, Wuhan, 430072, P. R. of China.}
\email{sunjun@whu.edu.cn}


\thanks {The first author is supported by National Key R$\&$D Program of China 2022YFA1005400 and NFSC No.12031017. The second author is supported by NFSC No. 11721101, 12031017. The third author is supported by NSFC No. 12071352, 12271039.}

\begin{document}

	\title[Translating solitons to Lagrangian mean curvature flow]
	{Translating solitons to a Lagrangian mean curvature flow with zero Maslov class}

	\begin{abstract}
It is known that there is no a Type I singularity for the Lagrangian mean curvature flow with zero Maslov class. In this paper, we study translating solitons which are important models of Type II singularities. A necessary condition for a blow-up limit arising at a Type II singularity of a Lagrangian mean curvature flow with zero Maslov class is provided. As an application, we try to understand the important open question proposed by Joyce-Lee-Tsui \cite{JLT} and Neves-Tian \cite{NT}, whether the Lagrangian translating solitons constructed by Joyce-Lee-Tsui can be a blow-up limit for a Lagrangian mean curvature flow with zero Maslov class.
	\end{abstract}
	
	\maketitle


	\section{Introduction}
	
	\allowdisplaybreaks
	
	\vspace{.1in}
	
\noindent Let $(M,\bar\omega,J,\bar g,\Omega)$ be a Calabi-Yau manifold of complex dimension $n$ with K\"ahler form $\bar\omega$, compatible complex structure $J$, associated Riemannian metric $\bar g$ and parallel holomorphic $n$-form $\Omega$. An $n$-dimensional oriented submanifold $\Sigma$ immersed into $M$ is called {\it Lagrangian submanifold} if $\bar\omega|_{\Sigma}=0$. For a Lagrangian submanifold, the induced area form $d\mu_{\Sigma}$ on $\Sigma$ and the holomorphic $n$-form $\Omega$ are related by (\cite{HarLaw}):
	\begin{equation}\label{e-angle}
		\Omega|_\Sigma=e^{i\theta}d\mu_{\Sigma}=\cos\theta d\mu_\Sigma+i\sin\theta d\mu_\Sigma,
	\end{equation}
where $\theta$ is multi-valued and is well-defined up to an additive constant $2k\pi$, $k\in {\mathbb Z}$. Nevertheless, $\cos\theta$ and $\sin\theta$ are single-valued functions on $\Sigma$. The function $\theta$ is called {\it Lagrangian angle} of $\Sigma$.

A Lagrangian submanifold $\Sigma$ is called {\it Special Lagrangian} if $\theta$ is a constant; {\it almost calibrated} if $\cos\theta>0$. We say that $\Sigma$ is {\it Lagrangian with zero Maslov class} if $\theta$ is a single-valued function on $\Sigma$. It is known that the Lagrangian angle and the mean curvature vector of $\Sigma$ are related by (\cite{TY})
		\begin{equation}\label{e-Htheta}
{\bf H}=J\nabla\theta.
		\end{equation}
In particular, a connected Lagrangian submanifold is special if and only if it is minimal.

	\vspace{.1in}

The existence of special Lagrangian submanifolds is an important problem in differential geometry and mathematical physics. In a compact K\"ahler-Einstein surface, Schoen and Wolfson have shown the existence of a branched surface which minimizes areas among Lagrangian competitors in each Lagrangian homology class by variational methods (\cite{ScW}).
Since a special Lagrangian submanifold is minimal, it makes it possible to approach the existence via the mean curvature flow.

Let $(M^{n+k},\bar g)$ be a Riemannian manifold and $\Sigma^n$ be a manifold. Given an initial immersion
	$$F_0:\Sigma\to M ,~~~~\Sigma_0=F_0(\Sigma),$$
	we consider a one-parameter family of smooth maps
	$$F_t=F(\cdot
	,t):\Sigma\rightarrow M$$
	with corresponding images
	$$\Sigma_t=F_t(\Sigma)$$
	satisfying the mean curvature flow equation (MCF):
	\begin{equation}\label{meaneqn}
		\left\{\begin{array}{lll}
			  \frac{d}{dt}F(x,t)={\bf H}(x,t),\\
			  \,\,\,\,\,F(x,0)=F_0(x).
		\end{array}\right.
	\end{equation}
	Here ${\bf H}(x,t)$ is the mean curvature vector of $\Sigma_t$ at $F(x,t)$ in $M$.

When the ambient manifold $M$ is a K\"ahler-Einstein manifold of complex dimension $n$ and the initial submanifold $\Sigma$ is Lagrangian, Smoczyk (\cite{Smo}) proved that along the mean curvature flow, the Lagrangian property is preserved and the Lagrangian angle $\theta$ satisfies the evolution equation
		\begin{equation}\label{e-LMCF-theta}
			\frac{\partial }{\partial t}\theta=\Delta\theta.
		\end{equation}
In this case, the flow is called {\it Lagrangian mean curvature flow}. From (\ref{e-LMCF-theta}), it is easy to see that ``almost calibrated Lagrangian" and ``zero Maslov class" properties are preserved along the mean curvature flow. These flows are called {\it almost calibrated Lagrangian mean curvature flow} and {\it Lagrangian mean curvature flow with zero Maslov class}, respectively.

		\vspace{.1in}
	
	In general, the mean curvature flow may develop a singularity if the maximum of the norm of the second fundamental form blows up (\cite{Huisken1}). In \cite{Neves}, Neves constructed Lagrangians with arbitrarily small Lagrangian angles Hamiltonian isotopic to a plane in ${\mathbb C}^n$, which
develop finite time singularities along the mean curvature flow. Several years later, Neves constructed a Lagrangian mean curvature flow in a compact Calabi-Yau surface which develops a finite time singularity (\cite{Neves2}). Recently, Su-Tsai-Wood constructed infinite time singularities of the Lagrangian mean curvature flow (\cite{STW}). These examples show that singularities of the Lagrangian mean curvature flow are generally unavoidable. 

The counterexamples constructed by Neves are Lagrangians with zero Maslov classes, but not almost calibrated. Actually, it was conjectured by Thomas and Yau that

	\vspace{.1in}

\noindent \textbf{Conjecture (\cite{TY}):} {\em Let L be a flow-stable Lagrangian in a Calabi-Yau manifold. Then the Lagrangian mean curvature flow will exist for all time and converge to the unique special Lagrangian in its Hamiltonian isotopy class.}

	\vspace{.1in}

	 According to the blow-up rate of the second fundamental form, Huisken classified the finite time singularities of the mean curvature flow into two types (\cite{Huisken2}). With the help of the monotonicity formula (\cite{Huisken2}),  Huisken proved that the smooth blow-up limits of a Type I singularity consist of self-shrinkers.  By establishing a new monotonicity formula, Chen-Li (\cite{CL2}) proved that there is no Type I singularity for an almost calibrated Langrangian mean curvature flow (Wang (\cite{Wang1}) independently proved it). In \cite{CL2}, Chen-Li showed further that the tangent flow of an  almost calibrated Lagrangian mean curvature flow is a stationary tangent cone. Neves proved that these results also hold for zero-Maslov class case (\cite{Neves}).

	Since there are no Type I singularities, it is crucial to study finite time Type II singularities.  By the standard blow-up analysis (\cite{HuS2}), it is known that the blow-up limit of a Type II singularity of the mean curvature flow consists of an eternal solution to the mean curvature flow in ${\mathbb C}^n$ with uniformly bounded second fundamental form. 
	
	\vspace{.1in}

	As a typical Type II singularity, translating soliton plays an important role in the study of the mean curvature flow.
	
	\begin{definition}
		A submanifold $\Sigma^n$ in ${\mathbb R}^{n+k}$ is a \textbf{translating soliton} to the mean curvature flow if there is a constant vector ${\bf T}\in {\mathbb R}^{n+k}$, such that
		\begin{align*}
			\Sigma_t=\Sigma+t{\bf T}
		\end{align*}
		is a solution to the mean curvature flow. Equivalently, it holds
		\begin{equation}\label{e-TS}
			{\bf T}^{\perp}={\bf H}
		\end{equation}
		on $\Sigma$ everywhere, where ${\bf H}$ is the mean curvature vector of $\Sigma$ in ${\mathbb R}^{n+k}$.
	\end{definition}

There are a few results on the Type II blow-up limits for the Lagrangian mean curvature flow (\cite{HL2}, \cite{HLS}, \cite{HS}, \cite{LS}, etc.). For instance, we proved that any eternal solution to the mean curvature flow with nonnegative Gauss curvature cannot be a blow-up limit of an almost calibrated Lagrangian mean curvature flow. Neves-Tian proved a rigidity result for Lagrangian translating solitons (\cite{NT}). They showed that under some natural assumptions, if the first Betti number of the Lagrangian translating soliton is finite, and either $\int_{\Sigma}|{\bf H}|^2d\mu$ is finite or $\Sigma$ is static and almost calibrated, then the translating soliton is a plane.

In (\cite{JLT}), Joyce-Lee-Tsui constructed interesting examples for Lagrangian translating solitons with oscillation of the Lagrangian angles arbitrarily small (see Proposition \ref{prop-JLT} in Section 2). These examples are Lagrangian with zero Maslov class but not almost calibrated. It is an interesting question that whether the example can arise as a blow-up limit of the Lagrangian mean curvature flow. This question was proposed by Joyce-Lee-Tsui (\cite{JLT}, Question 3.4) and Neves-Tian (\cite{NT}, 1.1 Open questions).
	\vspace{.1in}

\noindent \textbf{Open Question:} (\cite{JLT}, \cite{NT}) Can the translating solitons with small Lagrangian angle oscillation constructed by Joyce-Lee-Tsui in Proposition \ref{prop-JLT} arise as blow-ups of finite time singularities for the Lagrangian mean curvature flow, particularly when $n=2$?

		\vspace{.1in}

Let $X_0\in M$ be the blow-up point of the Lagrangian mean curvature flow with zero Maslov class at the first singular time $T$ and $\Sigma$ be a translating soliton arising as blow-up limit of the flow at $(X_0,T)$ (we refere to Section 3 and Section 4 for the details of the blow-up procedure). Denote $\underline{\theta}$ and $\overline{\theta}$ the infimum and supremum of the Lagrangian angle $\theta$ on $\Sigma$, respectively. Then we have
\begin{equation}\label{e-blowup} 
\underline{\theta}=\lim_{r\to0}\lim_{t\to T}\inf_{s\in [t,T)}\inf_{B_r(X_0)\cap \Sigma_s}\theta.
\end{equation}
	
	\vspace{.1in}
	
	Our main purpose in this paper is to understand the above important open question, and prove the following result, which may be seen as   a partial answer to the above open question:
	
	\vspace{.1in}

\begin{theorem}\label{thm-1}
Let $\{\Sigma_t\}_{t\in [0,T)}$ be a smooth solution to the Lagrangian mean curvature flow with zero Maslov class which develops a singularity at a finite time $T$. Assume that there exist positive constants $r_0$ and $t_1<T$, such that
\begin{equation}\label{e-blowup-assum} 
\underline{\theta}=\inf_{s\in [t_1,T)}\inf_{B_{r_0}(X_0)\cap \Sigma_s}\theta,
\end{equation}
and
\begin{equation}\label{e-main-assum-2-0} 
\int_{t_1}^T\int_{\Sigma_t}\log(A-\log|\theta-\underline{\theta}|)d\mu_tdt<\infty,
\end{equation}
where $A$ is a positive constant such that $A-\log|\theta-\underline{\theta}|>0$, then the translating soliton $\Sigma$ constructed by Joyce-Lee-Tsui in Proposition \ref{prop-JLT} cannot be a blow-up limit at $(X_0,T)$ of the flow for $n\geq 2$.
\end{theorem}

	\vspace{.1in}

\begin{remark}
Denote 
$$\underline{\theta}(t):=\inf_{\Sigma_t}\theta.$$ 
Then (\ref{e-LMCF-theta}) implies that $\underline{\theta}(t)$ is monotonically increasing. Therefore, the main assumption (\ref{e-blowup-assum}) roughly says that the singularity of the flow occurs away from the minimal points of the Lagrangian angle.
\end{remark}

	\vspace{.1in}

The idea in the proof of the above theorem can be described as follows: we establish a new monotonicity formula for the Lagrangian mean curvature flow with zero Maslov class. Applying it we obtain a necessary condition for a blow-up limit of the flow (see Theorem \ref{thm-necessary-LMCF} and Corollary \ref{cor-necessary-TS}). Then we prove the main theorem by a contradiction, which follows by constructing suitable test functions in the necessary condition. 

	\vspace{.1in}
	
	The subsequent sections are organized as follows: in Section 2, we recall Joyce-Lee-Tsui's example of Lagrangian translating solitons with zero Maslov classes; in Section 3, we establish a new monotonicity formula for the Lagrangian mean curvature flow with zero Maslov class; in Section 4, we derive the necessary condition for a blow-up limit of the flow; in Section 5, we prove the main result in this paper.
	
		\vspace{.1in}

\textbf{Acknowlegement:} The authors would like to thank Dr. Yuchen Bi for pointing out a gap in the first version of the paper and thank Dr. Zichang Liu for his nice suggestions after his careful reading.

	\vspace{.2in}

	\section{Joyce-Lee-Tsui's Example of Lagrangian Translating Solitons}
	
	\vspace{.1in}
	
	In this section, we will recall Joyce-Lee-Tsui's example of Lagrangian translating solitons with zero Maslov classes.

	\vspace{.1in}
	
In \cite{JLT}, Joyce-Lee-Tsui constructed  very interesting examples for Lagrangian translating solitons with oscillations of the Lagrangian angles arbitrarily small. Their construction is as follows (Corollary I of \cite{JLT}):

\begin{proposition}\label{prop-JLT}
For given constants $\alpha\geq 0$ and $a_1,\cdots,a_{n-1}>0$, define

\begin{equation}\label{e-phi}
\phi_j(y)=\int_0^y\frac{dt}{\left(\frac{1}{a_j}+t^2\right)\sqrt{P(t)}} ~,
\end{equation}
where
		\begin{equation*}
P(t)=\frac{1}{t^2}\left[ \prod_{k=1}^{n-1}(1+a_kt^2)\cdot e^{\alpha t^2}-1\right],
		\end{equation*}
for $j=1,2,\cdots,n-1$ and $y\in {\mathbb R}$. Then when $\alpha\neq 0$,
\begin{eqnarray}\label{e-L}
L&=&\left\{\left(x_1\sqrt{\frac{1}{a_1}+y^2}e^{i\phi_1(y)},\cdots,x_{n-1}\sqrt{\frac{1}{a_{n-1}}+y^2}e^{i\phi_{n-1}(y)},\frac{y^2- \sum_{j=1}^{n-1}x_j^2}{2}\right.\right.\nonumber\\
& &\left.\left. -\frac{i}{\alpha}\left[\sum_{j=1}^{n-1}\phi_j(y)+\arg\left(y+iP(y)^{-\frac{1}{2}}\right)\right]\right): x_1,\cdots,x_{n-1},y\in {\mathbb R}\right\}
\end{eqnarray}
is a closed, embedded Lagrangian in ${\mathbb C}^n$ diffeomorphic to ${\mathbb R}^n$, which is a Lagrangian translating soliton with translating vector  ${\bf T}=(0,0,\cdots,0,0,\alpha,0)\in {\mathbb R}^{2n}$.

There exist $\bar\phi_1,\cdots,\bar\phi_{n-1}\in (0,\frac{\pi}{2})$ such that $\phi_j(y)\to\bar\phi_j$ as $y\to\infty$ and $\phi_j(y)\to-\bar\phi_j$ as $y\to-\infty$ for $j=1,\cdots,n-1$, and $\bar\phi_1+\cdots+\bar\phi_{n-1}<\frac{\pi}{2}$. For fixed $\alpha>0$, the map $(a_1,\cdots,a_{n-1})\mapsto (\bar\phi_1,\cdots,\bar\phi_{n-1})$ is a 1-1 correspondence from $(0,\infty)^{n-1}$ to $\{(\bar\phi_1,\cdots,\bar\phi_{n-1})\in(0,\frac{\pi}{2})^{n-1}:\bar\phi_1+\cdots+\bar\phi_{n-1}<\frac{\pi}{2}\}$. 

The Lagrangian angle of $L$ in (\ref{e-L}) varies between $ \sum_{j=1}^{n-1}\bar\phi_j$ and $\pi- \sum_{j=1}^{n-1}\bar\phi_j$. In particular, by choosing $ \sum_{j=1}^{n-1}\bar\phi_j$ close to $\frac{\pi}{2}$, the oscillation of the Lagrangian angle can be made arbitrarily small.
\end{proposition}

\vspace{.1in}

For our later use, we will compute all the geometric quantities for the translating soliton $L$ in the above proposition.

The submanifold is parametrized by
\begin{eqnarray}\label{e-FF}
F(x_1,\cdots,x_{n-1},y)
&=&\left(x_1\sqrt{\frac{1}{a_1}+y^2}\cos\phi_1(y),x_1\sqrt{\frac{1}{a_1}+y^2}\sin\phi_1(y),\cdots,\right.\nonumber\\
& &\left. x_{n-1}\sqrt{\frac{1}{a_{n-1}}+y^2}\cos\phi_{n-1}(y),x_{n-1}\sqrt{\frac{1}{a_{n-1}}+y^2}\sin\phi_{n-1}(y),\right.\nonumber\\
& &\left.\frac{y^2- \sum_{j=1}^{n-1}x_j^2}{2},-\frac{1}{\alpha}\left[ \sum_{j=1}^{n-1}\phi_j(y)+\gamma(y)\right]\right),
\end{eqnarray}
where
\begin{equation}\label{e-gamma-2}
\gamma(y)=\arg\left(y+iP(y)^{-\frac{1}{2}}\right)=
\begin{cases}
   \arctan\frac{1}{\sqrt{ \prod_{k=1}^{n-1}(1+a_ky^2)\cdot e^{\alpha y^2}-1}}, & {\rm if} \ y>0;\\
    \frac{\pi}{2},  & {\rm if} \ y=0;\\
   -\arctan\frac{1}{\sqrt{ \prod_{k=1}^{n-1}(1+a_ky^2)\cdot e^{\alpha y^2}-1}}+\pi, & {\rm if} \ y<0.\\
\end{cases}
\end{equation}
By direct computations, we have
\begin{equation}\label{e-evo-phi-j}
\frac{d\phi_j}{dy}=
\begin{cases}
   \frac{y}{(\frac{1}{a_j}+y^2)\sqrt{ \prod_{k=1}^{n-1}(1+a_ky^2)\cdot e^{\alpha y^2}-1}}, & {\rm if} \ y> 0;\\
   -\frac{y}{(\frac{1}{a_j}+y^2)\sqrt{ \prod_{k=1}^{n-1}(1+a_ky^2)\cdot e^{\alpha y^2}-1}}, & {\rm if} \ y<0,
\end{cases}
\end{equation}
and
\begin{equation}\label{e-evo-gamma}
\frac{d\gamma}{dy}=
\begin{cases}
   -\frac{\alpha y}{\sqrt{ \prod_{k=1}^{n-1}(1+a_ky^2)\cdot e^{\alpha y^2}-1}}-\frac{ \sum_{j=1}^{n-1}\frac{y}{\frac{1}{a_j}+y^2}}{\sqrt{ \prod_{k=1}^{n-1}(1+a_ky^2)\cdot e^{\alpha y^2}-1}}, & {\rm if} \ y> 0;\\
   \frac{\alpha y}{\sqrt{ \prod_{k=1}^{n-1}(1+a_ky^2)\cdot e^{\alpha y^2}-1}}+\frac{ \sum_{j=1}^{n-1}\frac{y}{\frac{1}{a_j}+y^2}}{\sqrt{ \prod_{k=1}^{n-1}(1+a_ky^2)\cdot e^{\alpha y^2}-1}}, & {\rm if} \ y<0.
\end{cases}
\end{equation}
Hence we have
\begin{equation}\label{e-evo-theta-2}
\frac{d}{dy}\left(\sum_{j=1}^{n-1}\phi_j(y)+\gamma(y)\right)=
\begin{cases}
   -\frac{\alpha y}{\sqrt{ \prod_{k=1}^{n-1}(1+a_ky^2)\cdot e^{\alpha y^2}-1}}, & {\rm if} \ y> 0;\\
   \frac{\alpha y}{\sqrt{ \prod_{k=1}^{n-1}(1+a_ky^2)\cdot e^{\alpha y^2}-1}}, & {\rm if} \ y<0.
\end{cases}
\end{equation}

For any $y$, we have
\begin{equation}\label{e-monotone}
    \frac{d\phi_j}{dy}>0,\ \  \frac{d\gamma}{dy}<0, \ \ \frac{d}{dy}\left(\sum_{j=1}^{n-1}\phi_j(y)+\gamma(y)\right)<0.
\end{equation}
We also have by definition that
\begin{equation*}
    \lim_{y\to\infty}\phi_j(y)=\bar\phi_j, \ \lim_{y\to-\infty}\phi_j(y)=-\bar\phi_j,\ \lim_{y\to\infty}\gamma(y)=0, \ \lim_{y\to-\infty}\gamma(y)=\pi.
\end{equation*}
Therefore, we obtain
\begin{equation}\label{e-limits}
    \lim_{y\to\infty}\left(\sum_{j=1}^{n-1}\phi_j(y)+\gamma(y)\right)=\sum_{j=1}^{n-1}\bar\phi_j, \ \lim_{y\to-\infty}\left(\sum_{j=1}^{n-1}\phi_j(y)+\gamma(y)\right)=\pi-\sum_{j=1}^{n-1}\bar\phi_j.
\end{equation}

	\vspace{.1in}

To proceed further, we first consider the case that $y>0$. In this case, it is clear that the tangent space of $L$ is spanned by
		\begin{equation*}
F_{x_1}=\left(\sqrt{\frac{1}{a_1}+y^2}\cos\phi_1(y), \sqrt{\frac{1}{a_1}+y^2}\sin\phi_1(y), 0, 0, \cdots, 0, 0, -x_1, 0\right)^T,
		\end{equation*}
  	\begin{equation*}
F_{x_2}=\left( 0, 0, \sqrt{\frac{1}{a_2}+y^2}\cos\phi_2(y), \sqrt{\frac{1}{a_2}+y^2}\sin\phi_2(y), \cdots, 0, 0, -x_2, 0\right)^T,
		\end{equation*}
    	\begin{equation*}
\cdots
        \end{equation*}
  	\begin{equation*}
F_{x_{n-1}}=\left(0, 0, 0, 0, \sqrt{\frac{1}{a_{n-1}}+y^2}\cos\phi_{n-1}(y), \sqrt{\frac{1}{a_{n-1}}+y^2}\sin\phi_{n-1}(y), \cdots, -x_{n-1}, 0\right)^T
		\end{equation*}
and
\begin{equation*}
F_y=\left(\begin{array}{c}
  \frac{x_1y}{\sqrt{\frac{1}{a_1}+y^2}}\left(\cos\phi_1(y)-\frac{\sin\phi_1(y)}{\sqrt{ \prod_{k=1}^{n-1}(1+a_ky^2)\cdot e^{\alpha y^2}-1}}\right) \\
  \frac{x_1y}{\sqrt{\frac{1}{a_1}+y^2}}\left(\sin\phi_1(y)+\frac{\cos\phi_1(y)}{\sqrt{ \prod_{k=1}^{n-1}(1+a_ky^2)\cdot e^{\alpha y^2}-1}}\right)\\
  \vdots\\
  \frac{x_{n-1}y}{\sqrt{\frac{1}{a_{n-1}}+y^2}}\left(\cos\phi_{n-1}(y)-\frac{\sin\phi_{n-1}(y)}{\sqrt{ \prod_{k=1}^{n-1}(1+a_ky^2)\cdot e^{\alpha y^2}-1}}\right) \\
  \frac{x_{n-1}y}{\sqrt{\frac{1}{a_{n-1}}+y^2}}\left(\sin\phi_{n-1}(y)+\frac{\cos\phi_{n-1}(y)}{\sqrt{ \prod_{k=1}^{n-1}(1+a_ky^2)\cdot e^{\alpha y^2}-1}}\right)\\
  y\\
  \frac{y}{\sqrt{ \prod_{k=1}^{n-1}(1+a_ky^2)\cdot e^{\alpha y^2}-1}}\\
\end{array}
\right).
\end{equation*}
The induced metric on $L$ is given by
\begin{equation*}
g_{x_{j}x_{j}}=\langle F_{x_{j}},F_{x_{j}}\rangle=\frac{1}{a_j}+x_j^2+y^2,
\end{equation*}
\begin{equation*}
g_{x_i,x_j}=\langle F_{x_i},F_{x_j}\rangle=x_ix_j,\ \ 
g_{x_j,y}=\langle F_{x_j},F_y\rangle=0,
\end{equation*}
\begin{equation*}
g_{yy}=\langle F_y,F_y\rangle=\left(\sum_{j=1}^{n-1}\frac{x_j^2}{\frac{1}{a_j}+y^2}+1\right)\frac{ \prod_{k=1}^{n-1}(1+a_ky^2)\cdot e^{\alpha y^2}y^2}{ \prod_{k=1}^{n-1}(1+a_ky^2)\cdot e^{\alpha y^2}-1}.
\end{equation*}
In other words,
\begin{eqnarray*}
g&=&\left(
   \begin{array}{cccc}
      g_{x_1x_1} & \cdots & g_{x_1x_{n-1}} &  g_{x_1y}  \\
      \vdots & \ddots & \vdots & \vdots \\
      g_{x_{n-1}x_1} & \cdots & g_{x_{n-1}x_{n-1}} &  g_{x_{n-1}y}  \\
      g_{yx_1} & \cdots & g_{yx_{n-1}} &  g_{yy}  \\
   \end{array}
\right)\\
&=&\left(
   \begin{array}{cc}
      g_1 & 0 \\
      0 &  \left(\sum_{j=1}^{n-1}\frac{x_j^2}{\frac{1}{a_j}+y^2}+1\right)\frac{ \prod_{k=1}^{n-1}(1+a_ky^2)\cdot e^{\alpha y^2}y^2}{ \prod_{k=1}^{n-1}(1+a_ky^2)\cdot e^{\alpha y^2}-1}   \\
   \end{array}
\right),
\end{eqnarray*}
where
\begin{eqnarray}\label{e-g-1}
g_1=\left(
   \begin{array}{ccc}
      \frac{1}{a_1}+x_1^2+y^2 & \cdots & x_1x_{n-1}  \\
      \vdots & \ddots & \vdots  \\
      x_{n-1}x_1 & \cdots & \frac{1}{a_{n-1}}+x_{n-1}^2+y^2   \\
   \end{array}
\right).
\end{eqnarray}
It is easy to see that
\begin{equation}\label{E-detg-1}
{\rm det}g_1=\left(\sum_{j=1}^{n-1}\frac{x_j^2}{\frac{1}{a_j}+y^2}+1\right)\cdot \prod_{k=1}^{n-1}\left(\frac{1}{a_k}+y^2\right),
\end{equation}
and so
\begin{equation}\label{E-detg}
{\rm det}g=\left(\sum_{j=1}^{n-1}\frac{x_j^2}{\frac{1}{a_j}+y^2}+1\right)^2\cdot \prod_{k=1}^{n-1}\left(\frac{1}{a_k}+y^2\right)^2\cdot\frac{ \prod_{k=1}^{n-1}a_k\cdot e^{\alpha y^2}y^2}{ \prod_{k=1}^{n-1}(1+a_ky^2)\cdot e^{\alpha y^2}-1}.
\end{equation}
In particular, the area element of $L$ with respect to the induced metric is given by
\begin{equation}\label{e-dmu}
d\mu=\frac{\left( \sum_{j=1}^{n-1}\frac{x_j^2}{\frac{1}{a_j}+y^2}+1\right)\cdot  \prod_{k=1}^{n-1}\left(\frac{1}{a_k}+y^2\right)\cdot\sqrt{ \prod_{k=1}^{n-1}a_k\cdot e^{\alpha y^2}y^2}}{\sqrt{ \prod_{k=1}^{n-1}(1+a_ky^2)\cdot e^{\alpha y^2}-1}}dx_1\cdots dx_{n-1}dy.
\end{equation}

\vspace{.1in}

For the standard complex structure
\begin{equation*}
J_0=\left(\begin{array}{ccccccc}
    0 & -1  &  &  &  &  & \\
    1 & 0  &  &  &  &  & \\
      &    & 0 & -1 &  &   &  \\
      &    & 1 & 0 &   &   &  \\
      &    &   &   & \ddots &   &  \\
      &    &   &   &   & 0 & -1\\
      &    &   &   &   & 1 & 0\\
\end{array}
\right)
\end{equation*}
on ${\mathbb C}^n$, we set
		\begin{equation*}
\nu_{x_1}=J_0F_{x_1}=\left(-\sqrt{\frac{1}{a_1}+y^2}\sin\phi_1(y), \sqrt{\frac{1}{a_1}+y^2}\cos\phi_1(y), 0, 0,\cdots,0, 0, 0,-x_1\right)^T,
		\end{equation*}
		\begin{equation*}
\nu_{x_2}=J_0F_{x_2}=\left(0, 0, -\sqrt{\frac{1}{a_2}+y^2}\sin\phi_2(y), \sqrt{\frac{1}{a_2}+y^2}\cos\phi_2(y), \cdots,0, 0, 0,-x_2\right)^T,
		\end{equation*}
		\begin{equation*}
\cdots
		\end{equation*}		\begin{eqnarray*}
&&\nu_{x_{n-1}}=J_0F_{x_{n-1}}\\
&=&\left(0, 0, 0, 0, \cdots,-\sqrt{\frac{1}{a_{n-1}}+y^2}\sin\phi_{n-1}(y), \sqrt{\frac{1}{a_{n-1}}+y^2}\cos\phi_{n-1}(y),0,-x_{n-1}\right)^T,
		\end{eqnarray*}
\begin{equation*}
\nu_y=J_0F_y=\left(\begin{array}{c}
    -\frac{x_1y}{\sqrt{\frac{1}{a_1}+y^2}}\left(\sin\phi_1(y)+\frac{\cos\phi_1(y)}{\sqrt{ \prod_{k=1}^{n-1}(1+a_ky^2)\cdot e^{\alpha y^2}-1}}\right)\\
    \frac{x_1y}{\sqrt{\frac{1}{a_1}+y^2}}\left(\cos\phi_1(y)-\frac{\sin\phi_1(y)}{\sqrt{ \prod_{k=1}^{n-1}(1+a_ky^2)\cdot e^{\alpha y^2}-1}}\right) \\
  \vdots\\
    -\frac{x_{n-1}y}{\sqrt{\frac{1}{a_{n-1}}+y^2}}\left(\sin\phi_{n-1}(y)+\frac{\cos\phi_{n-1}(y)}{\sqrt{ \prod_{k=1}^{n-1}(1+a_ky^2)\cdot e^{\alpha y^2}-1}}\right)\\
   \frac{x_{n-1}y}{\sqrt{\frac{1}{a_{n-1}}+y^2}}\left(\cos\phi_{n-1}(y)-\frac{\sin\phi_{n-1}(y)}{\sqrt{ \prod_{k=1}^{n-1}(1+a_ky^2)\cdot e^{\alpha y^2}-1}}\right) \\
  -\frac{y}{\sqrt{ \prod_{k=1}^{n-1}(1+a_ky^2)\cdot e^{\alpha y^2}-1}}\\
  y\\
\end{array}
\right).
\end{equation*}

By direct computations, we see that $\nu_{x_1}$, $\nu_{x_2}$, $\nu_{x_{n-1}}$, $\nu_{y}$ are normal to $F_{x_{1}}$, $F_{x_{2}}$, $\cdots$, $F_{x_{n-1}}$, $F_y$ so that $J_0$ maps $TL$ onto $NL$. This implies that $L$ is a Lagrangian submanifold with respect to the complex structure $J_0$.
Actually, the Lagrangian angle of $L$ is given by (\cite{JLT})
\begin{equation}\label{e-theta}
\theta(x,y)=\theta(y)=\sum_{j=1}^{n-1}\phi_j(y)+\gamma(y)
\end{equation}
From (\ref{e-evo-theta-2}), we see that
\begin{equation}\label{e-evo-theta}
\frac{d\theta}{dy}=
\begin{cases}
   -\frac{\alpha y}{\sqrt{ \prod_{k=1}^{n-1}(1+a_ky^2)\cdot e^{\alpha y^2}-1}}, & {\rm if} \ y> 0;\\
   \frac{\alpha y}{\sqrt{ \prod_{k=1}^{n-1}(1+a_ky^2)\cdot e^{\alpha y^2}-1}}, & {\rm if} \ y<0.
\end{cases}
\end{equation}
In particular,
\begin{equation*}
\frac{d\theta}{dy}<0,
\end{equation*}
and it follows from (\ref{e-limits}) that
\begin{equation*}
\lim_{y\to\infty}\theta(y)=\sum_{j=1}^{n-1}\bar\phi_j, \  \theta(0)=\frac{\pi}{2}, \ \lim_{y\to-\infty}\theta(y)=\pi-\sum_{j=1}^{n-1}\bar\phi_j.
\end{equation*}

Finally, we compute the norm of the mean curvature vector. We will denote by $g^{-1}=(g^{ij})$ the inverse matrix of $g$. Then it is easy to see that
\begin{eqnarray*}
g^{-1}=\left(
   \begin{array}{cc}
      g_1^{-1} & 0 \\
      0 &  \frac{ \prod_{k=1}^{n-1}(1+a_ky^2)\cdot e^{\alpha y^2}-1}{\left( \sum_{j=1}^{n-1}\frac{x_j^2}{\frac{1}{a_j}+y^2}+1\right)\cdot  \prod_{k=1}^{n-1}(1+a_ky^2)\cdot e^{\alpha y^2}y^2}   \\
   \end{array}
\right),
\end{eqnarray*}
where $g_1^{-1}$ is the inverse matrix of $g_1$ given by (\ref{e-g-1}). 

Since $L$ is a translating soliton with translating vector ${\bf T}=(0,0,\cdots,0,0,\alpha,0)\in {\mathbb R}^{2n}$, we see that $\langle {\bf T}, \nu_{x_j}\rangle=0$ for $1\leq j\leq n-1$ so that ($x_n=y$)
\begin{eqnarray}\label{e-H}
  |\nabla \theta|^2&=& |{\bf H}|^2= |{\bf T}^{\perp}|^2\nonumber\\
  &=&\sum_{i,j=1}^{n}g^{ij}\langle {\bf T}, \nu_{x_i}\rangle \langle {\bf T}, \nu_{x_j}\rangle=g^{yy}\langle {\bf T}, \nu_{y}\rangle^2\nonumber\\
   &=&\frac{ \prod_{k=1}^{n-1}(1+a_ky^2)\cdot e^{\alpha y^2}-1}{\left( \sum_{j=1}^{n-1}\frac{x_j^2}{\frac{1}{a_j}+y^2}+1\right)\cdot  \prod_{k=1}^{n-1}(1+a_ky^2)\cdot e^{\alpha y^2}y^2} \cdot \frac{\alpha^2 y^2}{ \prod_{k=1}^{n-1}(1+a_ky^2)\cdot e^{\alpha y^2}-1}\nonumber\\
   &=&\frac{\alpha^2}{\left( \sum_{j=1}^{n-1}\frac{x_j^2}{\frac{1}{a_j}+y^2}+1\right)\cdot  \prod_{k=1}^{n-1}(1+a_ky^2)\cdot e^{\alpha y^2}}.
\end{eqnarray}

	\vspace{.2in}

\section{Monotonicity Formula for Lagrangian Mean Curvature Flow}

	\vspace{.1in}

In this section, we will derive a general monotonicity formula for the Lagrangian mean curvature flow with zero Maslov class. 

	\vspace{.1in}

Let $f\in C^1({\mathbb R})$ be a nonnegative function of $\theta$ with the second derivative piecewise continuous. Then we have from (\ref{e-LMCF-theta}) that
		\begin{equation}\label{e-LMCF-f}
		\left(\frac{\partial }{\partial t}-\Delta\right)f(\theta)=-f''|{\bf H}|^2.
		\end{equation}

	\vspace{.1in}

Let $i_M$ be the injectivity radius of $M$. Fix any point $X_0\in M$ and $t_0>0$, we choose a cutoff function $\phi\in C_0^{\infty}(B_{2r}(X_0))$ with $\phi\equiv 1$ in $B_{r}(X_0)$, where $0<2r<i_M$. Choose a normal coordinate system in $B_{2r}(X_0)$ and express the immersion $F$ using the coordinates $(F^1, \cdots, F^{2n})$ as a submanifold in ${\mathbb R}^{2n}$. 

Now we assume $X$ is the position vector of $\Sigma_t$ in ${\mathbb R}^{2n}$. Let $H(X,X_0,t,t_0)$ be the backward heat kernel on ${\mathbb R}^{2n}$ and set
\begin{equation*}
\rho(X,X_0,t,t_0)=(4\pi(t_0-t))^{\frac{n}{2}}H(X,X_0,t,t_0)=\frac{1}{(4\pi(t_0-t))^{\frac{n}{2}}}\exp\left(-\frac{|X-X_0|^2}{4(t_0-t)}\right),
		\end{equation*}
for $t<t_0$. Define 
\begin{equation*}
\Phi_f(X_0,t_0,t)=\int_{\Sigma_t}f(\theta)\phi(F)\rho(X,X_0,t,t_0)d\mu_t.
		\end{equation*}
Then by direct computations, we have the following weighted monotonicity formula for the Lagrangian mean curvature flow with zero Maslov class (see \cite{CL2} for an almost calibrated Lagrangian mean curvature flow with $f(\theta)=\frac{1}{\cos\theta}$):

\begin{proposition}\label{prop-monoto} For any $\varepsilon>0$, there exist positive constants $c_1$ and $c_2$  depending on $M$, $F_0$ and $r$ where $r$ is the constant in the definition of $\phi$, such that for any nonnegative function $f\in C^1({\mathbb R})$ with the second derivative piecewise continuous and $\varepsilon\in (0,1)$, we have
\begin{eqnarray}\label{e-monotonicity}
\frac{d}{dt}\left(e^{c_1\sqrt{t_0-t}}\Phi_f(X_0,t_0,t)\right)
&\leq&   -e^{c_1\sqrt{t_0-t}}\int_{\Sigma_t}f\phi\rho(X,X_0,t,t_0)\left|{\bf H}+\frac{(X-X_0)^{\perp}}{2(t_0-t)}\right|^2d\mu_t\nonumber\\
& &  -e^{c_1\sqrt{t_0-t}}\int_{\Sigma_t}\left(f''-\varepsilon^2f\right)|{\bf H}|^2\phi\rho(X,X_0,t,t_0)d\mu_t\nonumber\\
& &+\frac{c_2}{\varepsilon^2}e^{c_1\sqrt{t_0-t}}\int_{\Sigma_t\cap{\rm supp}\phi}fd\mu_t.
\end{eqnarray}
\end{proposition}

\begin{proof}
Note that
\begin{equation*}
\Delta F={\bf H}+g^{ij}\Gamma_{ij}^{\alpha}v_{\alpha},
\end{equation*}
where $\{v_{\alpha}\}_{\alpha=n+1,\cdots ,2n}$ is a basis of $N\Sigma_t$, $g_{ij}$ is the induced metric on $\Sigma$, $(g^{ij})$ is the inverse of $(g_{ij})$ and $\Gamma_{ij}^{\alpha}$ is the Christoffel symbol on $M$. Then we have
\begin{eqnarray}\label{e-rho-2}
\left(\frac{\partial}{\partial t}+\Delta\right)\rho(X,X_0,t,t_0)
&=&-\left(\frac{\langle X-X_0,{\bf H}+g^{ij}\Gamma_{ij}^{\alpha}v_{\alpha}\rangle}{t_0-t}+\frac{|(X-X_0)^{\perp}|^2}{4(t_0-t)^2}\right)\rho\nonumber\\
&=&-\left(\left|{\bf H}+\frac{(X-X_0)^{\perp}}{2(t_0-t)}\right|^2-|{\bf H}|^2+\frac{\langle X-X_0,g^{ij}\Gamma_{ij}^{\alpha}v_{\alpha}\rangle}{t_0-t}\right)\rho.
\end{eqnarray}
We also have
\begin{equation*}
\frac{\partial}{\partial t}d\mu_t=-|{\bf H}|^2d\mu_t, \ \ \frac{\partial}{\partial t}\phi(F)=\langle D\phi,{\bf H}\rangle.
\end{equation*}
We compute using (\ref{e-LMCF-f}) that
\begin{eqnarray*}
\frac{d}{dt}\Phi_f(X_0,t_0,t)
&=& \frac{d}{dt}\int_{\Sigma_t}f(\theta)\phi(F)\rho(X,X_0,t,t_0) d\mu_t\nonumber\\
&=& \int_{\Sigma}\phi\rho\left(\frac{\partial}{\partial t}-\Delta\right)f d\mu_t+\int_{\Sigma}\phi f\left(\frac{\partial}{\partial t}+\Delta\right)\rho d\mu_t\nonumber\\
&&+\int_{\Sigma_t}\left(\frac{\partial}{\partial t}\phi\right)f\rho d\mu_t-\int_{\Sigma}f\phi\rho |{\bf H}|^2 d\mu_t\nonumber\\
&&+\int_{\Sigma_t}\phi\rho\Delta fd\mu_t-\int_{\Sigma_t}\phi f\Delta \rho d\mu_t\nonumber\\
&\leq& -\int_{\Sigma}\phi \rho\left\{f''|{\bf H}|^2+f\left|{\bf H}+\frac{(X-X_0)^{\perp}}{2(t_0-t)}\right|^2\right\} d\mu_t\nonumber\\
&& +\int_{\Sigma_t}\langle D\phi,{\bf H}\rangle f\rho d\mu_t+\int_{\Sigma_t}f\rho \Delta\phi d\mu_t+2\int_{\Sigma_t}f\langle\nabla\rho,\nabla\phi\rangle d\mu_t\nonumber\\
&&-\int_{\Sigma_t}\phi f\rho \frac{\langle X-X_0,g^{ij}\Gamma_{ij}^{\alpha}v_{\alpha}\rangle}{t_0-t}d\mu_t.
\end{eqnarray*}
As in the proof of Proposition 2.1 in \cite{CL2} (see (13) in \cite{CL2}), we see that
\begin{equation}
\left|\frac{\langle X-X_0,g^{ij}\Gamma_{ij}^{\alpha}v_{\alpha}\rangle}{t_0-t}\right|\rho\leq  c_1\frac{\rho(X,X_0,t,t_0)}{2\sqrt{t_0-t}}+C.
\end{equation}
Notice that $D\phi\equiv0$ in $B_r(X_0)$. So $|\rho \Delta\phi|$ and $\langle\nabla\rho,\nabla\phi\rangle$ are bounded in $B_{2r}(X_0)$. Hence
\begin{equation*}
\int_{\Sigma_t}|f\rho \Delta\phi| d\mu_t+2\int_{\Sigma_t}|f\langle\nabla\rho,\nabla\phi\rangle| d\mu_t\leq C\int_{\Sigma_t}fd\mu_t,
\end{equation*}
where $C$ depends only on $r$ and $\max\left(|D^2\phi|^2+|D\phi|^2\right)$.
We also have that for any $\varepsilon\in (0,1)$
\begin{equation*}
|\langle D\phi,{\bf H}\rangle|\leq \varepsilon^2\phi |{\bf H}|^2+\frac{1}{4\varepsilon^2}\frac{|D\phi|^2}{\phi}.
\end{equation*}
Furthermore, since $\phi\in C_0^{\infty}(B_{2r}(X_0), {\mathbb R}^+)$, we have
\begin{equation*}
\frac{|D\phi|^2}{\phi}\leq 2\max_{\phi>0}|D^2\phi|^2.
\end{equation*}
Hence we have
\begin{eqnarray}\label{e-Psif}
\frac{d}{dt}\Phi_f(X_0,t_0,t)
&\leq& -\int_{\Sigma_t}f\phi\rho(X,X_0,t,t_0)\left|{\bf H}+\frac{(X-X_0)^{\perp}}{2(t_0-t)}\right|^2d\mu_t\nonumber\\
& &  -\int_{\Sigma_t}\left(f''-\varepsilon^2f\right)|{\bf H}|^2\phi\rho(X,X_0,t,t_0)d\mu_t\nonumber\\
& &+\frac{c_1}{2\sqrt{t_0-t}}\Phi_f+\frac{C}{\varepsilon^2}\int_{\Sigma_t\cap{\rm supp}\phi}fd\mu_t.
\end{eqnarray}
Therefore, we obtain 
\begin{eqnarray*}
\frac{d}{dt}\left(e^{c_1\sqrt{t_0-t}}\Phi_f(X_0,t_0,t)\right)
&\leq&  -e^{c_1\sqrt{t_0-t}}\int_{\Sigma_t}f\phi\rho(X,X_0,t,t_0)\left|{\bf H}+\frac{(X-X_0)^{\perp}}{2(t_0-t)}\right|^2d\mu_t\nonumber\\
& &  -e^{c_1\sqrt{t_0-t}}\int_{\Sigma_t}\left(f''-\varepsilon^2f\right)|{\bf H}|^2\phi\rho(X,X_0,t,t_0)d\mu_t\nonumber\\
& &+\frac{c_2}{\varepsilon^2}e^{c_1\sqrt{t_0-t}}\int_{\Sigma_t\cap{\rm supp}\phi}fd\mu_t.
\end{eqnarray*}
This completes the proof of the proposition.
\end{proof}

	\vspace{.1in}

When the ambient manifold is the Euclidean space ${\mathbb C}^n$ and suppose that the Lagrangian mean curvature flow with zero Maslov class develops finite time singularity at $(X_0, T)$ for some $X_0$ in ${\mathbb C}^n$, then we can take the cutoff function $\phi\equiv 1$. 
Checking the proof of the above proposition carefully, we obtain that

\begin{corollary}\label{cor-monoto} Let $\Sigma_t$ be a Lagrangian mean curvature flow with zero Maslov class in ${\mathbb C}^n$. suppose that the flow with zero Maslov class develops finite time singularity at $(X_0, T)$ for some $X_0\in {\mathbb C}^n$. Then for any nonnegative function $f\in C^1({\mathbb R})$ with the second derivative piecewise continuous, we have
\begin{eqnarray}\label{e-monotonicity-cn}
\frac{d}{dt}\int_{\Sigma_t}f(\theta)\rho(X,X_0,t,t_0)d\mu_t
&=&   -\int_{\Sigma_t}f\rho(X,X_0,t,t_0)\left|{\bf H}+\frac{(X-X_0)^{\perp}}{2(t_0-t)}\right|^2d\mu_t\nonumber\\
& &  -\int_{\Sigma_t}f''|{\bf H}|^2\rho(X,X_0,t,t_0)d\mu_t.
\end{eqnarray}
\end{corollary}

	\vspace{.2in}

\section{Blow-up Analysis of Lagrangian Mean Curvature Flows}

	\vspace{.1in}

In this section, we will provide a necessary condition on an eternal mean curvature flow that arises as a blow-up limit at a finite time singularity of the Lagrangian mean curvature flow with zero Maslov class.

	\vspace{.1in}

Using the above new monotonicity formula with the choice of $f(\theta)=\theta^2$, we can show that there is no a finite time Type I singularity for the Lagrangian mean curvature flow with zero Maslov class (\cite{Neves}). Therefore we may assume that the flow develops a Type II singularity at the first finite singular time $T$. Namely, 
\begin{equation*}
\limsup_{t\to T}\left[(T-t)\max_{\Sigma_t}|{\bf A}|^2\right]=\infty.
\end{equation*}

Following \cite{HuS2}, we choose a sequence $\{x_k,t_k\}$ as follows. For any integer $k\geq 1$, let $t_k\in [0,T-\frac{1}{k}]$, $x_k\in \Sigma$ such that
\begin{equation}\label{e-sff-max}
|{\bf A}|^2(x_k,t_k)\left(T-\frac{1}{k}-t_k\right)=\max_{x\in \Sigma, t\leq T-\frac{1}{k}}|{\bf A}|^2(x,t)\left(T-\frac{1}{k}-t\right).
\end{equation}
Furthermore, set
\begin{equation}
L_k=|{\bf A}|(x_k,t_k), \ \ \alpha_k=-L_k^{2}t_k, \ \ \omega_k=L_k^2\left(T-\frac{1}{k}-t_k\right).
\end{equation}
Because the singularity is of Type II, we see from Lemma 4.3 of \cite{HuS2} that
\begin{equation}
t_k\to T, \ L_k\to\infty, \ \alpha_k\to-\infty, \ \ \omega_k\to\infty
\end{equation}
as $k\to\infty$.

Since $M$ is closed, without loss of generality, we may assume that $F(x_k,t_k):=X_k\to \tilde{X}_0$. We will choose a normal coordinate system around $\tilde{X}_0$ so that we can view the flow around $\tilde{X}_0$ as a family of submanifolds in ${\mathbb R}^{2n}$.

Now for any $k\geq 1$ we consider the rescaled flow
\begin{equation}
\tilde{\Sigma}^k_{\tau}:=L_k(\Sigma_{L_k^{-2}\tau+t_k}-X_k)
\end{equation}
given by
\begin{equation*}
\tilde{F}_k(\cdot,\tau)=L_k(F(\cdot,L_k^{-2}\tau+t_k)-F(x_k,t_k)), \ \ \tau\in [\alpha_k,\omega_k].
\end{equation*}
Then, for each $k$, $\tilde{\Sigma}^k_{\tau}$ is a solution to the mean curvature flow with a uniformly bounded second fundamental form. By the standard argument we see that up to a subsequence, $\tilde{\Sigma}^k_{\tau}$ converges smoothly locally to a smooth complete nonflat eternal solution to the mean curvature flow $\tilde{\Sigma}^{\infty}_{\tau}$ in ${\mathbb R}^{2n}$ which is given by $\tilde{F}_\infty$. We also assume $\tilde{X}_k$ and $\tilde{X}_\infty$ are the position vectors of $\tilde{\Sigma}^{k}_{\tau}$ and $\tilde{\Sigma}^{\infty}_{\tau}$ in ${\mathbb R}^{2n}$, respectively. Furthermore, we have $|\tilde{\bf A}|\leq 1$ equal to 1 at some point. $\tilde{\Sigma}^{\infty}_{\tau}$ is called {\it blow-up limit} of the mean curvature flow.

	\vspace{.1in}

With these preparations, we can prove the following necessary condition for blow-up flows, especially for Lagrangian translating solitons that arise as blow-up limit of Lagrangian mean curvature flow with zero Maslov class:

	\vspace{.1in}

\begin{theorem}\label{thm-necessary-LMCF}
If $\tilde{\Sigma}_{\tau}^{\infty}$ is an eternal solution to the mean curvature flow in ${\mathbb C}^n$ arising as a blow-up limit at a finite time singularity of the Lagrangian mean curvature flow with zero Maslov class, then for any nonnegative function $f\in C^1({\mathbb R})$ with the second derivative piecewise continuous satisfying 
\begin{equation}\label{assum1}
f''\geq \varepsilon^2f
\end{equation}
for some $\varepsilon>0$, and any fixed $t_1<T$, there is a constant $c_3$ depending only on the initial surface $\Sigma_0$, $M$, $\varepsilon$ and $T$, such that for any ${\bf a}\in {\mathbb R}^{2n}$
\begin{eqnarray}\label{e-necessary-flow}
& & \int_{-\infty}^{0}\int_{\tilde{\Sigma}_{\tau}^{\infty}}(f''-\varepsilon^2f)|{\bf H}_{\tilde{\Sigma}_{\tau}^{\infty}}|^2\frac{1}{(-\tau)^{\frac{n}{2}}}\exp\left(\frac{|\tilde{X}_\infty-{\bf a}|^2}{4\tau}\right)d\tilde{\mu}^{\infty}_\tau d\tau\\
&\leq& c_3\left(\int_{t_1}^T\int_{\Sigma_t\cap{\rm supp}\phi}fd\mu_tdt+\frac{1}{(T-t_1)^{\frac{n}{2}}}\int_{\Sigma_{t_1}}\phi fd\mu_{t_1}\right).\nonumber
\end{eqnarray}
\end{theorem}

\begin{proof} 
For any $R>0$, we choose a cut-off function $\phi_R\in C^{\infty}_0(B_{2R}(0))$ with $\phi_R=1$ in $B_{R}(0)$, where $B_{\rho}(0)$ is the metric ball of radius $\rho$ centered at $0$ in ${\mathbb R}^{2n}$. Then from $\tilde{X}_k=L_k(X-X_k)$ and $t=L_k^{-2}\tau+t_k$, we have
\begin{eqnarray*}
& &\int_{\tilde{\Sigma}^k_{\tau}}f(\theta_k)\phi_R(\tilde{F}_k)\frac{1}{(-\tau)^{\frac{n}{2}}}\exp\left(\frac{|L_k^{-1}\tilde{X}_k+X_k-X_0|^2}{4L_k^{-2}\tau}\right)d\tilde{\mu}^{k}_\tau\\
&=&\int_{\Sigma_t}f(\theta)\phi(F)\frac{1}{(t_k-t)^{\frac{n}{2}}}\exp\left(-\frac{|X-X_0|^2}{4(t_k-t)}\right)d\mu_t,
\end{eqnarray*}
where $\theta_k(\tilde{F}_k)=\theta(F)$. 

Fixing any point ${\bf a}\in {\mathbb R}^{2n}$, we choose $X_0=X_k+L_k^{-1}{\bf a}$ which still lies in the small neighborhood of  $\tilde{X}_0$ for $k$ sufficiently large. Then we have from (\ref{e-monotonicity}) that for any $0<R_1<R_2$
\begin{eqnarray}\label{e-k-1}
& & e^{c_1L_k^{-1}\sqrt{R_2}}\Phi_f(X_k+L_k^{-1}{\bf a},t_k,t_k-L_k^{-2}R_2)\\
& & -e^{c_1L_k^{-1}\sqrt{R_1}}\Phi_f(X_k+L_k^{-1}{\bf a},t_k,t_k-L_k^{-2}R_1)\nonumber\\
&\geq&   \int_{t_k-L_k^{-2}R_2}^{t_k-L_k^{-2}R_1}e^{c_1\sqrt{t_k-t}}\nonumber\\
& & \cdot\left(\int_{\Sigma_t}f(v)\phi\rho(X,X_k+L_k^{-1}{\bf a},t,t_k)\left|{\bf H}+\frac{(X-X_k-L_k^{-1}{\bf a})^{\perp}}{2(t_k-t)}\right|^2d\mu_t\right.\nonumber\\
& & \left.+ \int_{\Sigma_t}(f''-\varepsilon^2f)|{\bf H}|^2\phi\rho(X,X_k+L_k^{-1}{\bf a},t,t_k)d\mu_t\right)dt\nonumber\\
& & - \int_{t_k-L_k^{-2}R_2}^{t_k-L_k^{-2}R_1}\frac{c_2}{\varepsilon^2}e^{c_1\sqrt{t_k-t}}\int_{\Sigma_t\cap{\rm supp}\phi}fd\mu_tdt\nonumber\\
&\geq&   \int_{t_k-L_k^{-2}R_2}^{t_k-L_k^{-2}R_1}e^{c_1\sqrt{t_k-t}}\nonumber\\
& & \cdot \int_{\Sigma_t}(f''-\varepsilon^2f)|{\bf H}|^2\phi\rho(X,X_k+L_k^{-1}{\bf a},t,t_k)d\mu_tdt\nonumber\\
& & - \int_{t_k-L_k^{-2}R_2}^{t_k-L_k^{-2}R_1}\frac{c_2}{\varepsilon^2}e^{c_1\sqrt{t_k-t}}\int_{\Sigma_t\cap{\rm supp}\phi}fd\mu_tdt\nonumber\\
&\geq& \int_{-R_2}^{-R_1}e^{c_1L_k^{-1}\sqrt{-\tau}}\int_{\tilde{\Sigma}_{\tau}^k}(f''-\varepsilon^2f)|{\bf H}|^2\phi_R(\tilde{F}_k)\frac{1}{(-4\pi \tau)^{\frac{n}{2}}}\exp\left(\frac{|\tilde{X}_k-{\bf a}|^2}{4\tau}\right)d\tilde{\mu}^{k}_\tau d\tau\nonumber\\
& &-L_k^{-2-\frac{n}{2}}\int_{-R_2}^{-R_1}e^{c_1L_k^{-1}\sqrt{-\tau}}\int_{\tilde{\Sigma}_{\tau}^k\cap B_{L_kr}}fd\tilde{\mu}^{k}_\tau d\tau.\nonumber
\end{eqnarray}

\vspace{.1in}

\noindent\textbf{Claim:} For any $t_1<T$, there exists a constant $c_3$ independent of $R_2$ and $k$ such that for $k$ sufficiently large, 
\begin{eqnarray*}
& & e^{c_1L_k^{-1}\sqrt{R_2}}\Phi_f(X_k+L_k^{-1}{\bf a},t_k,t_k-L_k^{-2}R_2)\\
&\leq& \frac{c_3}{(4\pi)^{\frac{n}{2}}}\left(\int_{t_1}^T\int_{\Sigma_t\cap{\rm supp}\phi}fd\mu_tdt+\frac{1}{(T-t_1)^{\frac{n}{2}}}\int_{\Sigma_{t_1}}\phi fd\mu_{t_1}\right).
\end{eqnarray*}

\vspace{.1in}

\noindent\textbf{Proof of the claim:} From (\ref{e-monotonicity}), we see that
\begin{equation*}
\frac{d}{dt}\left(e^{c_1\sqrt{t_k-t}}\Phi_f(X_k+L_k^{-1}{\bf a},t_k,t)\right)\leq \frac{c_2}{\varepsilon^2}e^{c_1\sqrt{t_k-t}}\int_{\Sigma_t\cap{\rm supp}\phi}fd\mu_t.
		\end{equation*}
Therefore, for $k$ sufficiently large such that $t_k-t_1>\frac{1}{2}(T-t_1)$ and $t_k-L_k^{-2}R_2>t_1$ we obtain that
\begin{eqnarray*}
& & e^{c_1L_k^{-1}\sqrt{R_2}}\Phi_f(X_k+L_k^{-1}{\bf a},t_k,t_k-L_k^{-2}R_2)\\
&=& e^{c_1\sqrt{t_k-(t_k-L_k^{-2}R_2)}}\Phi_f(X_k+L_k^{-1}{\bf a},t_k,t_k-L_k^{-2}R_2)\\
&\leq&  e^{c_1\sqrt{t_k-t_1}}\Phi_f\left(X_k+L_k^{-1}{\bf a},t_k,t_1\right)\\
& & +\int_{t_1}^{t_k-L_k^{-2}R_2} \frac{c_2}{\varepsilon^2}e^{c_1\sqrt{t_k-t}}\int_{\Sigma_t\cap{\rm supp}\phi}fd\mu_tdt\\
&\leq& e^{c_1\sqrt{T}}\int_{\Sigma_{t_1}}f(\theta)\phi(F)\frac{1}{[4\pi(t_k-t_1)]^{\frac{n}{2}}}\exp\left(-\frac{|X-(X_k-L_k^{-1}{\bf a})|^2}{4(t_k-t_1)}\right)d\mu_{t_1}\\
& &+\frac{c_2}{\varepsilon^2}e^{c_1\sqrt{T}}\int_{t_1}^{T} \int_{\Sigma_t\cap{\rm supp}\phi}fd\mu_tdt\\
&\leq& e^{c_1\sqrt{T}}\frac{1}{[4\pi(t_k-t_1)]^{\frac{n}{2}}}\int_{\Sigma_{t_1}}\phi fd\mu_{t_1}+\frac{c_2}{\varepsilon^2}e^{c_1\sqrt{T}}\int_{t_1}^{T} \int_{\Sigma_t\cap{\rm supp}\phi}fd\mu_tdt\\
&\leq&\frac{c_2}{\varepsilon^2}e^{c_1\sqrt{T}}\int_{t_1}^T\int_{\Sigma_t\cap{\rm supp}\phi}fd\mu_tdt+\frac{e^{c_1\sqrt{T}}}{(2\pi)^{\frac{n}{2}}(T-t_1)^{\frac{n}{2}}}\int_{\Sigma_{t_1}}\phi fd\mu_{t_1}.
\end{eqnarray*}
The claim then follows from taking $c_3=\left[\frac{e^{c_1\sqrt{T}}}{(2\pi)^{\frac{n}{2}}}+\frac{c_2}{\varepsilon^2}e^{c_1\sqrt{T}}\right](4\pi)^{\frac{n}{2}}$.
\hfill $\square$

\vspace{.1in}

By the above claim, letting $k\to\infty$ in (\ref{e-k-1}) yields for any $0<R_1<R_2$ 
\begin{eqnarray*}
&&\int_{-R_2}^{-R_1}\int_{\tilde{\Sigma}_{\tau}^{\infty}\cap B_R(0)}(f''-\varepsilon^2f)|{\bf H}_{\tilde{\Sigma}_{\tau}^{\infty}}|^2\frac{1}{(-\tau)^{\frac{n}{2}}}\exp\left(\frac{|\tilde{X}_\infty-{\bf a}|^2}{4\tau}\right)d\tilde{\mu}^{\infty}_\tau d\tau\\
&\leq& c_3\left(\int_{t_1}^T\int_{\Sigma_t\cap{\rm supp}\phi}fd\mu_tdt+\frac{1}{(T-t_1)^{\frac{n}{2}}}\int_{\Sigma_{t_1}}\phi fd\mu_{t_1}\right).
\end{eqnarray*}
Letting $R_1\to0$, $R_2\to\infty$ first and $R\to\infty$ next, we see that
\begin{eqnarray*}\label{e-infty}
&&\int_{-\infty}^{0}\int_{\tilde{\Sigma}_{\tau}^{\infty}}(f''-\varepsilon^2f)|{\bf H}_{\tilde{\Sigma}_{\tau}^{\infty}}|^2\frac{1}{(-\tau)^{\frac{n}{2}}}\exp\left(\frac{|\tilde{X}_\infty-{\bf a}|^2}{4\tau}\right)d\tilde{\mu}^{\infty}_\tau d\tau\\
&\leq& c_3\left(\int_{t_1}^T\int_{\Sigma_t\cap{\rm supp}\phi}fd\mu_tdt+\frac{1}{(T-t_1)^{\frac{n}{2}}}\int_{\Sigma_{t_1}}\phi fd\mu_{t_1}\right)
\end{eqnarray*}
holds for any fixed ${\bf a}\in {\mathbb R}^4$.
\end{proof}

	\vspace{.1in}

As a consequence, we can obtain the necessary condition for a translating soliton to be a blow-up limit at a finite time singularity of the Lagrangian mean curvature flow with zero Maslov class:

\begin{corollary}\label{cor-necessary-TS}
If $\Sigma$ is a Lagrangian translating soliton arising as a blow-up limit of a finite time singularity of the Lagrangian mean curvature flow with zero Maslov class, then for any nonnegative function $f\in C^1({\mathbb R})$ with the second derivative piecewise continuous satisfying 
\begin{equation}\label{assum2}
f''\geq \varepsilon^2f
\end{equation}
for some $\varepsilon>0$ and any fixed $t_1<T$, there is a constant $c_4$ depending only on the initial surface $\Sigma_0$, $M$, $\varepsilon$, $T$ and ${\bf T}$, such that for any ${\bf a}\in {\mathbb R}^{2n}$
\begin{eqnarray}\label{e-necessary-TS}
& & \int_{\Sigma}(f''-\varepsilon^2f)\frac{|{\bf H}_{\Sigma}|^2}{|X-{\bf a}|^{\frac{n}{2}-1}}e^{\frac{\langle X-{\bf a}, {\bf T}\rangle}{2}}K_{\frac{n}{2}-1}\left(\frac{|{\bf T}||X-{\bf a}|}{2}\right) d\mu\\
&\leq& c_4\left(\int_{t_1}^T\int_{\Sigma_t\cap{\rm supp}\phi}fd\mu_tdt+\frac{1}{(T-t_1)^{\frac{n}{2}}}\int_{\Sigma_{t_1}}\phi fd\mu_{t_1}\right),\nonumber
\end{eqnarray}
where
\begin{eqnarray*}
K_{\nu}(z):=\frac{1}{2}\left(\frac{z}{2}\right)^{\nu}\int^{\infty}_{0}\frac{1}{t^{\nu+1}}\exp\left(-\frac{z^2}{4t}-t\right) dt
\end{eqnarray*}
is the modified Bessel function of the second kind. 
\end{corollary}

\begin{proof} 
Since $\Sigma$ is a translating soliton, we have $\tilde{\Sigma}^{\infty}_{\tau}=\Sigma+\tau {\bf T}$. Using the fact that $|{\bf H}_{\tilde{\Sigma}_{\tau}^{\infty}}|^2(X+\tau {\bf T})=|{\bf H}_{\Sigma}|^2(X)$, we have by Theorem \ref{thm-necessary-LMCF} that
\begin{eqnarray*}\label{e-k}
& & c_3\left(\int_{t_1}^T\int_{\Sigma_t\cap{\rm supp}\phi}fd\mu_tdt+\frac{1}{(T-t_1)^{\frac{n}{2}}}\int_{\Sigma_{t_1}}\phi fd\mu_{t_1}\right)\\
&\geq & \int_{-\infty}^{0}\int_{\tilde{\Sigma}_{\tau}^{\infty}}(f''-\varepsilon^2f)|{\bf H}_{\tilde{\Sigma}_{\tau}^{\infty}}|^2\frac{1}{(-\tau)^{\frac{n}{2}}}\exp\left(\frac{|\tilde{X}-{\bf a}|^2}{4\tau}\right)d\tilde{\mu}^{\infty}_\tau d\tau\nonumber\\
&=&  \int_{-\infty}^{0}\int_{\Sigma}(f''-\varepsilon^2f)|{\bf H}_{\Sigma}|^2\frac{1}{(-\tau)^{\frac{n}{2}}}\exp\left(\frac{|X+\tau {\bf T}-{\bf a}|^2}{4\tau}\right)d\mu d\tau\nonumber\\
&=&   \int^{\infty}_{0}\int_{\Sigma}(f''-\varepsilon^2f)|{\bf H}_{\Sigma}|^2 \cdot\frac{1}{\tau^{\frac{n}{2}}}\exp\left(-\frac{|X-\tau {\bf T}-{\bf a}|^2}{4\tau}\right)d\mu d\tau\nonumber\\
&=&   \int^{\infty}_{0}\int_{\Sigma}(f''-\varepsilon^2f)|{\bf H}_{\Sigma}|^2\frac{1}{\tau^{\frac{n}{2}}}\exp\left(-\frac{|X-{\bf a}|^2}{4\tau}-\frac{|{\bf T}|^2\tau}{4}+\frac{\langle X-{\bf a}, {\bf T}\rangle}{2}\right)d\mu d\tau\nonumber\\
&=&  \int_{\Sigma}(f''-\varepsilon^2f)|{\bf H}_{\Sigma}|^2e^{\frac{\langle X-{\bf a}, {\bf T}\rangle}{2}}\int^{\infty}_{0}\frac{1}{\tau^{\frac{n}{2}}}\exp\left(-\frac{|X-{\bf a}|^2}{4\tau}-\frac{|{\bf T}|^2\tau}{4}\right) d\tau d\mu\nonumber\\
&=&  2|{\bf T}|^{\frac{n}{2}-1}\int_{\Sigma}(f''-\varepsilon^2f)\frac{|{\bf H}_{\Sigma}|^2}{|X-{\bf a}|^{\frac{n}{2}-1}}e^{\frac{\langle X-{\bf a}, {\bf T}\rangle}{2}}K_{\frac{n}{2}-1}\left(\frac{|{\bf T}||X-{\bf a}|}{2}\right) d\mu,
\end{eqnarray*}
where
\begin{eqnarray*}
K_{\nu}(z):=\frac{1}{2}\left(\frac{z}{2}\right)^{\nu}\int^{\infty}_{0}\frac{1}{t^{\nu+1}}\exp\left(-\frac{z^2}{4t}-t\right) dt
\end{eqnarray*}
is the modified Bessel function of the second kind. This completes the proof of the theorem.
\end{proof}

	\vspace{.2in}

\section{Proof of the Main Theorems}

	\vspace{.1in}

In this section, we will first prove that a special function is integrable along the mean curvature flow by using the volume estimate by Jiang, Li and Wang  (\cite{JLW}). We then show that the example constructed by Joyce, Lee and Tsui (\cite{JLT}) cannot arise as blow up limit under some natural assumptions. This partially answers the open question proposed by Joyce-Lee-Tsui and Neves-Tian (see Scetion 1). The following theorem is a consequence of Theorem 1.8 of \cite{JLW}.

\begin{theorem}\label{thm-0}
Let $\{\Sigma_t\}_{0\leq t<T}$ be the Lagrangian mean curvature flow with zero Maslov class. Denote by $\underline{\theta}$ a fixed constant. Then for any $t\in [0,T)$, the following inequality holds
\begin{equation}\label{e-log}
\int_{\Sigma_t}\log(A-\log|\theta-\underline{\theta}|)d\mu_t\leq \Lambda(t)
\end{equation}
where $A$ is a positive constant such that $A-\log|\theta-\underline{\theta}|>0$. Here $\Lambda$ is a constant depending on $A$, $\Sigma$ and $t$.
\end{theorem}

\begin{proof}
Denote by $\theta_{max}:= \sup_{\Sigma_t}\theta$ and $\theta_{min}:= \inf_{\Sigma_t}\theta$. 
Then we have using the co-area formula and Theorem 1.8 of \cite{JLW} that
\begin{eqnarray*}
&&\int_{\Sigma_t}\log(A-\log|\theta-\underline{\theta}|)d\mu_t\\
&=&\int_{\Sigma_t}\log(A-\log((\theta-\underline{\theta})_++(\underline{\theta}-\theta)_+)d\mu_t\\
&=& \int_{\Sigma_t\cap\{\theta-\underline{\theta}\geq 0\}}\log(A-\log(\theta-\underline{\theta}))d\mu_t+\int_{\Sigma_t\cap\{\theta-\underline{\theta}\leq 0\}}\log(A-\log(\underline{\theta}-\theta))d\mu_t\\
&=&\int_{0}^{\theta_{max}-\underline{\theta}}\int_{\Sigma_t\cap\{\theta-\underline{\theta}=\alpha\}}\frac{\log(A-\log(\theta-\underline{\theta}))}{|\nabla\theta|}d\sigma_td\alpha \\
& & +\int^{0}_{\theta_{min}-\underline{\theta}}\int_{\Sigma_t\cap\{\theta-\underline{\theta}=\alpha\}}\frac{\log(A-\log(\underline{\theta}-\theta))}{|\nabla\theta|}d\sigma_td\alpha \\
&=& \int_{0}^{\theta_{max}-\underline{\theta}}\log(A-\log\alpha)d\left(\int_{\Sigma_t\cap\{0<\theta-\underline{\theta}<\alpha\}}d\mu_t\right)\\
& & +\int_{0}^{\underline{\theta}-\theta_{min}}\log(A-\log\alpha)d\left(\int_{\Sigma_t\cap\{0<\underline{\theta}-\theta<\alpha\}}d\mu_t\right)\\
&\leq & \log(A-\log(\theta_{max}-\underline{\theta}))\int_{\Sigma_t\cap\{0<\theta-\underline{\theta}<\theta_{max}-\underline{\theta}\}}d\mu_t\\
& & +\int_{0}^{\theta_{max}-\underline{\theta}}\left(\int_{\Sigma_t\cap\{0<\theta-\underline{\theta}<\alpha\}}d\mu_t\right)\frac{1}{\alpha(A-\log\alpha)}d\alpha\\
& & +\log(A-\log(\underline{\theta}-\theta_{min}))\int_{\Sigma_t\cap\{0<\underline{\theta}-\theta<\underline{\theta}-\theta_{min}\}}d\mu_t\\
& & +\int_{0}^{\underline{\theta}-\theta_{min}}\left(\int_{\Sigma_t\cap\{0<\underline{\theta}-\theta<\alpha\}}d\mu_t\right)\frac{1}{\alpha(A-\log\alpha)}d\alpha\\
&\leq& [\log(A-\log(\theta_{max}-\underline{\theta}))+\log(A-\log(\underline{\theta}-\theta_{min}))]{\rm Area}(\Sigma_t)\\
& &+\int_{0}^{\theta_{max}-\underline{\theta}}\frac{C\alpha^{\varepsilon}}{\alpha(A-\log\alpha)}d\alpha+\int_{0}^{\underline{\theta}-\theta_{min}}\frac{C\alpha^{\varepsilon}}{\alpha(A-\log\alpha)}d\alpha \\
&\leq & C(A,\theta_{max}-\underline{\theta},\underline{\theta}-\theta_{min},n,\Sigma_t,\theta) <\infty
\end{eqnarray*}
where $C$ and $\varepsilon$ are positive constants depending on $\Sigma_t$, $n$ and $\theta$. The last inequality follows from Theorem 1.8 in \cite{JLW}.
\end{proof}

\vspace{.1in}

For a solution $\{\Sigma_t\}_{t\in [0,T)}$ to the Lagrangian mean curvature flow with zero Maslov class in a compact Calabi-Yau manifold, we denote  $\theta$ as the Lagrangian angle of $\Sigma_t$. We also denote $\underline{\theta}(t)$ and $\overline{\theta}(t)$ as the minimum and the maximum of $\theta$ on $\Sigma_t$ respectively. Then $\underline{\theta}(t)$ and $\overline{\theta}(t)$ are Lipschitz and we have from (\ref{e-LMCF-theta}) that
		\begin{equation*}
			\frac{d}{dt}\underline{\theta}(t)\geq 0, \ \ \frac{d}{dt}\overline{\theta}(t)\leq 0.
		\end{equation*}
In particular, we have
		\begin{equation*}
\overline{\theta}(t)-\underline{\theta}(t)\leq\overline{\theta}(0)-\underline{\theta}(0):=D.
		\end{equation*}
Without loss of generality, we may assume that $D\leq 1$. On the Lagrangian translating soliton $\Sigma$, we will also denote  $\theta$ as the Lagrangian angle of the translating soliton $\Sigma$, $\underline\theta$ as its minimum and $\overline{\theta}$ as its  maximum. Since $\theta$ is uniformly bounded along the flow, we may assume that along the flow and consequently on the translating soliton we have
		\begin{equation*}
-D\leq \theta-\underline{\theta}\leq D.
		\end{equation*}

	\vspace{.1in}

Let $X_0\in M$ be the blow up point of the Lagrangian mean curvature flow with zero Maslov class at the first singular time $T$ and $\Sigma$ be a translating soliton arising as blow-up limit of the flow at $(X_0,T)$. Denote $\underline{\theta}$ and $\overline{\theta}$ the infimum and supremum of the Lagrangian angle $\theta$ on $\Sigma$, respectively. Then we have

\begin{equation}\label{e-blowup-2} 
\underline{\theta}=\lim_{r\to0}\lim_{t\to T}\inf_{s\in [t,T)}\inf_{B_r(X_0)\cap \Sigma_s}\theta.
\end{equation}
Now we will prove the main theorem in this paper.

	\vspace{.1in}

\begin{theorem}\label{Thm-1-2}
Let $\{\Sigma_t\}_{t\in [0,T)}$ be a smooth solution to the Lagrangian mean curvature flow with zero Maslov class and develop singularity at finite time $T$. Assume that there exist positive constants $r_0$ and $t_1<T$, such that
\begin{equation}\label{e-main-assum} 
\underline{\theta}=\inf_{s\in [t_1,T)}\inf_{B_{r_0}(X_0)\cap \Sigma_s}\theta,
\end{equation} 
and
\begin{equation}\label{e-main-assum-2} 
\int_{t_1}^T\int_{\Sigma_t}\log(A-\log|\theta-\underline{\theta}|)d\mu_tdt:=c_5<\infty,
\end{equation}
where $A$ is a positive constant such that $A-\log|\theta-\underline{\theta}|>0$, then the translating soliton $\Sigma$ constructed by Joyce-Lee-Tsui in Proposition \ref{prop-JLT} cannot be a blow-up limit at $(X_0,T)$ of the flow.
\end{theorem}

 \begin{proof} We prove the theorem by contradiction. Assume that the translating soliton $\Sigma$ constructed by Joyce-Lee-Tsui in Proposition \ref{prop-JLT} is a blow-up limit at $(X_0,T)$ of the flow.
 Without loss of generality, we may choose ${\bf T}=(0,0,\cdots,0,0,1,0)$, i.e.  $\alpha=1$. We will choose the cutoff function $\phi$ in Section 3 such that ${\rm supp}\phi\subset B_{r_0}(X_0)$. 
 By the assumption \ref{e-main-assum}, we see that for any $\delta\in (0,1)$ independent of $r_0$ and $t_1$, we have on $\Sigma_t\cap{\rm supp}\phi$ for $t\in [t_1,T)$ that
\begin{equation*}
    \theta>\underline{\theta}-\delta.
\end{equation*}
Set $v=\theta-\underline{\theta}$ and we define the function $f_{\delta}(v)$ by
\begin{equation}\label{E-f-n}
    f_{\delta}(v)=\log(A-\log(v+\delta)),
\end{equation}
where $A:=\log(D+1)+2$ is a positive constant.
Then $f_{\delta}$ is a smooth positive function on $\Sigma$, and $\phi f_{\delta}$ is a smooth function on $\Sigma_t$ for $t\in [t_1,T)$. We compute
\begin{eqnarray*}\label{E-f-n-1st}
f_{\delta}'(v)=
\frac{-1}{[A-\log(v+\delta)](v+\delta)},
\end{eqnarray*}
and
\begin{eqnarray}\label{E-f-n-2nd}
f_{\delta}''(v)=
\frac{1}{[A-\log(v+\delta)](v+\delta)^2}-\frac{1}{[A-\log(v+\delta)]^2(v+\delta)^2}.
\end{eqnarray}
By the choice of $A$, it is easy to see that the assumptions of Corollary \ref{cor-necessary-TS} are satisfied for some $\varepsilon$ depending on $D$. Therefore, by reducing $\varepsilon$ if necessary, we have the following inequality (by taking ${\bf a}={\bf 0}$):
\begin{eqnarray}\label{E-necessary-TS-2}
& & \int_{\Sigma}f_{\delta}''\frac{|{\bf H}_{\Sigma}|^2}{|X|^{\frac{n}{2}-1}}e^{\frac{\langle X, {\bf T}\rangle}{2}}K_{\frac{n}{2}-1}\left(\frac{|{\bf T}||X-{\bf a}|}{2}\right) d\mu\\
&\leq& c_6\left(\int_{t_1}^T\int_{\Sigma_t\cap{\rm supp}\phi}f_{\delta}d\mu_tdt+\frac{1}{(T-t_1)^{\frac{n}{2}}}\int_{\Sigma_{t_1}}\phi f_{\delta}d\mu_{t_1}\right).\nonumber
\end{eqnarray}

For the right hand side of (\ref{E-necessary-TS-2}), Theorem \ref{thm-0} and the assumption (\ref{e-main-assum-2}) implies that
\begin{eqnarray}\label{E-RHS}
\int_{t_1}^T\int_{\Sigma_t\cap{\rm supp}\phi}f_{\delta}d\mu_tdt+\frac{1}{(T-t_1)^{\frac{n}{2}}}\int_{\Sigma_{t_1}}\phi f_{\delta}d\mu_{t_1}\leq c_7
\end{eqnarray}
for some constant $c_7$ depending on $c_5$ and $t_1$, but independent of $\delta$.

In fact, without using Theorem 5.1, applying the assumption (\ref{e-main-assum-2}), we can see that (\ref{E-RHS}) holds for some $t_1$ in $(0,T)$.

We next estimate the left hand side of (\ref{E-necessary-TS-2}). By the choice of $A$, we see that
\begin{eqnarray}\label{E-f-n-2nd-2}
f_{\delta}''(v)\geq\frac{1}{2[A-\log(v+\delta)](v+\delta)^2}.
\end{eqnarray}
It remains to compute the quantity
\begin{eqnarray*}
\frac{|{\bf H}_{\Sigma}|^2}{|X|^{\frac{n}{2}-1}}e^{\frac{\langle X, {\bf T}\rangle}{2}}K_{\frac{n}{2}-1}\left(\frac{|{\bf T}||X-{\bf a}|}{2}\right)d\mu.
\end{eqnarray*}

\vspace{.1in}

From (\ref{e-FF}), it follows that
\begin{eqnarray}\label{E-X}
|X|^2=\left(\frac{ \sum_{j=1}^{n-1}x_j^2+y^2}{2}\right)^2+ \sum_{j=1}^{n-1}\frac{x_j^2}{a_j}+\left( \sum_{j=1}^{n-1}\phi_j(y)+\gamma(y)\right)^2.
\end{eqnarray}
Let $\theta$ be the Lagrangian angle of the translating soliton $\Sigma$ and $\underline\theta$ its infimum. Then it is known from Section 2 that for $y>0$
\begin{eqnarray*}
&&\theta(y)= \sum_{j=1}^{n-1}\phi_j(y)+\gamma(y)\\
&=&\sum_{j=1}^{n-1}\int_0^y\frac{tdt}{\left(\frac{1}{a_j}+t^2\right)\sqrt{ \prod_{i=1}^{n-1}(1+a_it^2)\cdot e^{\alpha t^2}-1}}+\arctan\frac{1}{\sqrt{ \prod_{i=1}^{n-1}(1+a_iy^2)\cdot e^{\alpha y^2}-1}},
\end{eqnarray*}
and
\begin{equation*}
\underline\theta=\lim_{y\to\infty}\theta(y)=\sum_{j=1}^{n-1}\bar\phi_j=\sum_{j=1}^{n-1}\int_0^{\infty}\frac{tdt}{\left(\frac{1}{a_j}+t^2\right)\sqrt{ \prod_{i=1}^{n-1}(1+a_it^2)\cdot e^{\alpha t^2}-1}}>0.
\end{equation*}
So,
\begin{eqnarray*}
&&v(y)=\theta(y)-\underline\theta\\
&=&\arctan\frac{1}{\sqrt{ \prod_{i=1}^{n-1}(1+a_iy^2)\cdot e^{\alpha y^2}-1}}-\sum_{j=1}^{n-1}\int_y^{\infty}\frac{tdt}{\left(\frac{1}{a_j}+t^2\right)\sqrt{ \prod_{i=1}^{n-1}(1+a_it^2)\cdot e^{\alpha t^2}-1}}.
\end{eqnarray*}
We therefore have $\lim\limits_{y\to\infty}v(y)=0$,
\begin{equation*}
\frac{dv}{dy}=-\frac{y}{\sqrt{ \prod_{j=1}^{n-1}(1+a_jy^2)\cdot e^{y^2}-1}},
\end{equation*}
and for $y>0$
\begin{equation}\label{E-v}
v(y)=\int_y^{\infty}\frac{t}{\sqrt{ \prod_{j=1}^{n-1}(1+a_jt^2)\cdot e^{t^2}-1}}dt.
\end{equation}

\vspace{.1in}

\begin{lemma}\label{Lem-theta}
We have for $y\geq 1$
\begin{equation}\label{E-v-ul}
\frac{1}{n\sqrt{ \prod_{j=1}^{n-1}(1+a_j)}}y^{1-n}e^{-\frac{y^2}{2}}\leq v(y)\leq \frac{1}{\sqrt{ \prod_{j=1}^{n-1}a_j}}y^{1-n}e^{-\frac{y^2}{2}}.
\end{equation}
\end{lemma}

\begin{proof}
By (\ref{E-v}), we compute by integration by parts that
\begin{eqnarray*}
& & v(y)\\
&=&\int_y^{\infty}\frac{t}{\sqrt{ \prod_{j=1}^{n-1}(1+a_jt^2)\cdot e^{t^2}-1}}dt
=\frac{1}{2}\int_{y^2}^{\infty}\frac{1}{\sqrt{ \prod_{j=1}^{n-1}(1+a_jt)\cdot e^{t}-1}}dt\\
&=&\frac{1}{2}\int_{y^2}^{\infty}\frac{e^{-\frac{t}{2}}}{\sqrt{ \prod_{j=1}^{n-1}(1+a_jt)-e^{-t}}}dt=-\int_{y^2}^{\infty}\frac{1}{\sqrt{ \prod_{j=1}^{n-1}(1+a_jt)-e^{-t}}}de^{-\frac{t}{2}}\\
&=&\frac{e^{-\frac{y^2}{2}}}{\sqrt{ \prod_{j=1}^{n-1}(1+a_jy^2)-e^{-y^2}}}-\frac{1}{2}\int_{y^2}^{\infty}\frac{e^{-\frac{t}{2}}\left( \sum_{i=1}^{n-1}a_i \prod_{j=1,j\neq i}^{n-1}(1+a_jt)+e^{-t}\right)}{\left( \prod_{j=1}^{n-1}(1+a_jt)-e^{-t}\right)^{\frac{3}{2}}}dt\\
&\leq&\frac{e^{-\frac{y^2}{2}}}{\sqrt{ \prod_{j=1}^{n-1}(1+a_jy^2)-e^{-y^2}}}
\leq \frac{1}{\sqrt{ \prod_{j=1}^{n-1}(1+a_j)}}y^{1-n}e^{-\frac{y^2}{2}}.
\end{eqnarray*}
To prove the lower bound, we compute for $y\geq 1$
\begin{eqnarray*}
v(y)
&=&\int_y^{\infty}\frac{t}{\sqrt{ \prod_{j=1}^{n-1}(1+a_jt^2)\cdot e^{t^2}-1}}dt\geq \int_y^{\infty}\frac{t}{\sqrt{ \prod_{j=1}^{n-1}(1+a_j)\cdot t^{2n-2}\cdot e^{t^2}}}dt\\
&=& \frac{1}{2}\left[ \prod_{j=1}^{n-1}(1+a_j)\right]^{-\frac{1}{2}}\int_{y^2}^{\infty}t^{\frac{1-n}{2}}e^{-\frac{t}{2}}dt=2^{\frac{1-n}{2}}\left[ \prod_{j=1}^{n-1}(1+a_j)\right]^{-\frac{1}{2}}\int_{\frac{y^2}{2}}^{\infty}s^{\frac{1-n}{2}}e^{-s}ds\\
&=&2^{\frac{1-n}{2}}\left[ \prod_{j=1}^{n-1}a_j\right]^{-\frac{1}{2}}\Gamma\left(\frac{3-n}{2},\frac{y^2}{2}\right),
\end{eqnarray*}
where $\Gamma(a,x)$ is the incomplete Gamma function defined by
\begin{equation*}
    \Gamma(a,x)=\int_{x}^{\infty}t^{a-1}e^{-t}dt.
\end{equation*}
Notice that we have the estimate for $x>0$ and $a<1$:
\begin{equation}\label{e-Gamma}
    \frac{x}{x+1-a}<x^{1-a}e^x\Gamma(a,x)<\frac{x+1}{x+2-a}.
\end{equation}
Therefore, we have for $y\geq 1$
\begin{equation*}
    v(y)\geq \left[ \prod_{j=1}^{n-1}(1+a_j)\right]^{-\frac{1}{2}}\frac{y^{3-n}}{y^2+n-1}e^{-\frac{y^2}{2}}\geq \frac{1}{n\sqrt{ \prod_{j=1}^{n-1}(1+a_j)}}y^{1-n}e^{-\frac{y^2}{2}}.
\end{equation*}
This proves the lemma.
\end{proof}

\vspace{.1in}

We next estimate $K_{\frac{n}{2}-1}\left(\frac{|X|}{2}\right)$. Since $\theta(y)= \sum_{j=1}^{n-1}\phi_j(y)+\gamma(y)
$ is nonincreasing for $y$, we see that
\begin{equation*}
0<\underline\theta\leq  \sum_{j=1}^{n-1}\phi_j(y)+\gamma(y)\leq \pi-\underline\theta.
\end{equation*}
Therefore, by (\ref{E-X}), we have
\begin{eqnarray}\label{E-X-2}
|X|^2\geq \left( \sum_{j=1}^{n-1}\phi_j(y)+\gamma(y)\right)^2\geq \underline\theta^2>0.
\end{eqnarray}
In particular, 
\begin{eqnarray*}
\frac{|X|}{2}\geq \frac{1}{2}\underline\theta>0.
\end{eqnarray*}

\begin{lemma}\label{Lem-K0}
For each $\nu$, there is a positive constant $c_8$ depending on $\underline\theta$, $n$ and $\nu$ such that 
\begin{equation}\label{E-K0}
\sqrt{z}e^zK_{\nu}(z)\geq c_8>0, \ \ \forall z\geq \frac{1}{2}\underline\theta.
\end{equation}
\end{lemma}

\begin{proof}
It is known that the asymptotic behaviour of $K_{\nu}(z)$ as $z\to \infty$ is given by
\begin{equation*}
K_{\nu}(z)\sim \sqrt{\frac{\pi}{2z}}e^{-z}\sum_{k=0}^{\infty}\frac{a_k(\nu)}{z^k}.
\end{equation*}
Therefore, there is a positive constant $c_9>0$ and $B_1$ sufficiently large, such that 
\begin{equation*}
\sqrt{z}e^zK_{\nu}(z)\geq c_9, \ \ \forall z\geq B_1.
\end{equation*}
Since the function $\sqrt{z}e^zK_{\nu}(z)$ is positive and continuous on $\left[\frac{1}{2}\underline\theta,B_1\right]$, there is $c_{10}>0$ such that
 \begin{equation*}
\sqrt{z}e^zK_{\nu}(z)\geq c_{10}, \ \ \forall z\in\left[\frac{1}{2}\underline\theta,B_1\right].
\end{equation*}
Choose $c_8=\min\{c_9,c_{10}\}$, and then we have
 \begin{equation*}
\sqrt{z}e^zK_{\nu}(z)\geq c_8>0, \ \ \forall z\in\left[\frac{1}{2}\underline\theta,\infty\right).
\end{equation*}
This proves the lemma.
\end{proof}

\vspace{.1in}

Lemma \ref{Lem-K0} tells us that for $y\geq 1$,
\begin{equation*}
K_{\frac{n}{2}-1}\left(\frac{|X|}{2}\right)\geq c_{8}\left(\frac{|X|}{2}\right)^{-\frac{1}{2}}e^{-\frac{|X|}{2}}.
\end{equation*}
We also have
\begin{equation*}
e^{\frac{\langle X, {\bf T}\rangle}{2}}=e^{\frac{y^2- \sum_{j=1}^{n-1}x_j^2}{4}}.
\end{equation*}
Therefore, we have from (\ref{e-dmu}) and (\ref{e-H}) that for $y\geq 1$
\begin{eqnarray*}
& & \frac{|{\bf H}_{\Sigma}|^2}{|X|^{\frac{n}{2}-1}}e^{\frac{\langle X, {\bf T}\rangle}{2}}K_{\frac{n}{2}-1}\left(\frac{|X|}{2}\right)d\mu\\
&=& e^{\frac{y^2-\sum_{j=1}^{n-1}x_j^2}{4}}|X|^{1-\frac{n}{2}}K_{\frac{n}{2}-1}\left(\frac{|X|}{2}\right)\cdot \frac{1}{\left( \sum_{j=1}^{n-1}\frac{x_j^2}{\frac{1}{a_j}+y^2}+1\right)\cdot  \prod_{j=1}^{n-1}(1+a_jy^2)\cdot e^{y^2}}\\
& & \cdot \frac{\left( \sum_{j=1}^{n-1}\frac{x_j^2}{\frac{1}{a_j}+y^2}+1\right)\cdot  \prod_{j=1}^{n-1}\left(\frac{1}{a_j}+y^2\right)\cdot\sqrt{ \prod_{j=1}^{n-1}a_j\cdot e^{y^2}y^2}}{\sqrt{ \prod_{j=1}^{n-1}(1+a_jy^2)\cdot e^{y^2}-1}}dx_1\cdots dx_{n-1}dy
\\
&=& \frac{e^{\frac{y^2- \sum_{j=1}^{n-1}x_j^2}{4}}|X|^{1-\frac{n}{2}}K_{\frac{n}{2}-1}\left(\frac{|X|}{2}\right)\cdot y}{\sqrt{ \prod_{j=1}^{n-1}a_j\cdot e^{y^2}}\cdot \sqrt{ \prod_{j=1}^{n-1}(1+a_jy^2)\cdot e^{y^2}-1}}dx_1\cdots dx_{n-1}dy
\\
&\geq& \frac{e^{\frac{y^2- \sum_{j=1}^{n-1}x_j^2}{4}}|X|^{1-\frac{n}{2}}\cdot c_8\left(\frac{|X|}{2}\right)^{-\frac{1}{2}}e^{-\frac{|X|}{2}}\cdot y}{\sqrt{ \prod_{j=1}^{n-1}a_j\cdot e^{y^2}}\cdot \sqrt{ \prod_{j=1}^{n-1}(1+a_jy^2)\cdot e^{y^2}-1}}dx_1\cdots dx_{n-1}dy\\
&=& \frac{\sqrt{2}c_8 e^{\frac{y^2- \sum_{j=1}^{n-1}x_j^2}{4}}|X|^{\frac{1-n}{2}}e^{-\frac{|X|}{2}}\cdot y}{\sqrt{ \prod_{j=1}^{n-1}a_j\cdot e^{y^2}}\cdot \sqrt{ \prod_{j=1}^{n-1}(1+a_jy^2)\cdot e^{y^2}-1}}dx_1\cdots dx_{n-1}dy\\
&=&\frac{\sqrt{2}c_8y\left[\left(\frac{ \sum_{j=1}^{n-1}x_j^2+y^2}{2}\right)^2+ \sum_{j=1}^{n-1}\frac{x_j^2}{a_j}+\left( \sum_{j=1}^{n-1}\phi_j(y)+\gamma(y)\right)^2\right]^{\frac{1-n}{4}}}{\sqrt{ \prod_{j=1}^{n-1}a_j}\cdot \sqrt{ \prod_{j=1}^{n-1}(1+a_jy^2)\cdot e^{y^2}-1}}\\
& & \cdot e^{-\frac{y^2+ \sum_{j=1}^{n-1}x_j^2}{4}-\frac{\sqrt{\left(\frac{ \sum_{j=1}^{n-1}x_j^2+y^2}{2}\right)^2+ \sum_{j=1}^{n-1}\frac{x_j^2}{a_j}+\left( \sum_{j=1}^{n-1}\phi_j(y)+\gamma(y)\right)^2}}{2}}dx_1\cdots dx_{n-1}dy\\
&\geq& c_{12}\frac{y^{2-n}}{\left( \sum_{j=1}^{n-1}x_j^2+y^2\right)^{\frac{n-1}{2}}}e^{- \sum_{j=1}^{n-1}x_j^2-y^2}dx_1\cdots dx_{n-1}dy\\
&\geq& c_{12}\frac{y^{3-2n}}{\left( \sum_{j=1}^{n-1}x_j^2+1\right)^{\frac{n-1}{2}}}e^{- \sum_{j=1}^{n-1}x_j^2-y^2}dx_1\cdots dx_{n-1}dy\\
&=& c_{12}\frac{e^{- \sum_{j=1}^{n-1}x_j^2}}{\left( \sum_{j=1}^{n-1}x_j^2+1\right)^{\frac{n-1}{2}}}y^{3-2n}e^{-y^2}dx_1\cdots dx_{n-1}dy.
\end{eqnarray*}
Here we have used the fact that for $y\geq 1$,
\begin{equation*}
\frac{y}{\sqrt{ \sum_{j=1}^{n-1}x_j^2+y^2}}\geq \frac{1}{\sqrt{ \sum_{j=1}^{n-1}x_j^2+1}}.
\end{equation*}
Therefore, we have
\begin{eqnarray}\label{E-necessary-TS-LHS}
& & \int_{\Sigma}f_{\delta}''\frac{|{\bf H}_{\Sigma}|^2}{|X|^{\frac{n}{2}-1}}e^{\frac{\langle X, {\bf T}\rangle}{2}}K_{\frac{n}{2}-1}\left(\frac{|{\bf T}||X-{\bf a}|}{2}\right) d\mu\geq c_{13}\int_1^{\infty}y^{3-2n}e^{-y^2}dy.
\end{eqnarray}
On the other hand, using (\ref{E-v}), we see that 
\begin{equation*}
dv=-\frac{y}{\sqrt{ \prod_{j=1}^{n-1}(1+a_jy^2)\cdot e^{y^2}-1}}dy.
\end{equation*}
Furthermore,  using Lemma \ref{Lem-theta}, we see that $\{v\leq s_0\}\subset\{y\geq 1\}$, where $s_0$ is given by 
\begin{equation}\label{e-s-0}
    s_0=\frac{1}{n\sqrt{ \prod_{j=1}^{n-1}(1+a_j)}}e^{-\frac{1}{2}}.
\end{equation}
Hence, we have
\begin{eqnarray}\label{E-necessary-TS-LHS-2}
& & \int_{\Sigma}f_{\delta}''\frac{|{\bf H}_{\Sigma}|^2}{|X|^{\frac{n}{2}-1}}e^{\frac{\langle X, {\bf T}\rangle}{2}}K_{\frac{n}{2}-1}\left(\frac{|{\bf T}||X-{\bf a}|}{2}\right) d\mu\\
&\geq& c_{14}\int_0^{s_0}\frac{y^{2-2n}e^{-y^2}\sqrt{ \prod_{j=1}^{n-1}(1+a_jy^2)\cdot e^{y^2}-1}}{2[A-\log(v+\delta)](v+\delta)^2}dv.\nonumber
\end{eqnarray}
Applying Lemma \ref{Lem-theta} again, we obtain for $y\geq 1$ that
\begin{eqnarray*}
y^{2-2n}e^{-y^2}\sqrt{ \prod_{j=1}^{n-1}(1+a_jy^2)\cdot e^{y^2}-1}
\sim y^{1-n}e^{-\frac{y^2}{2}} \sim v(y).
\end{eqnarray*}
Thus we get that
\begin{eqnarray}\label{E-necessary-TS-LHS-3}
& & \int_{\Sigma}f_{\delta}''\frac{|{\bf H}_{\Sigma}|^2}{|X|^{\frac{n}{2}-1}}e^{\frac{\langle X, {\bf T}\rangle}{2}}K_{\frac{n}{2}-1}\left(\frac{|{\bf T}||X-{\bf a}|}{2}\right) d\mu\\
&\geq& c_{15}\int_0^{s_0}\frac{v}{2[A-\log(v+\delta)](v+\delta)^2}dv\nonumber\\
&=& c_{15}\int_0^{s_0}\frac{1}{2[A-\log(v+\delta)](v+\delta)}dv-c_{15}\delta\int_0^{s_0}\frac{1}{2[A-\log(v+\delta)](v+\delta)^2}dv\nonumber\\
&=&\frac{c_{15}}{2}\log(A-\log\delta)-\frac{c_{15}}{2}\log(A-\log(s_0+\delta))\nonumber\\
& &-\frac{c_{15}\delta}{2}\int_0^{s_0}\frac{1}{[A-\log(v+\delta)](v+\delta)^2}dv.\nonumber
\end{eqnarray}
To estimate the last term, we compute using integrating by parts that
\begin{eqnarray*}
\int_0^{s_0}\frac{1}{[A-\log(v+\delta)](v+\delta)^2}dv
&=& \frac{1}{\delta(A-\log\delta)}-\frac{1}{(s_0+\delta)[A-\log(s_0+\delta)]}\\
&& +\int_0^{s_0}\frac{1}{[A-\log(v+\delta)]^2(v+\delta)^2}dv
\end{eqnarray*}
By the choice of $A$, we get that
\begin{eqnarray*}
&& \frac{1}{\delta(A-\log\delta)}-\frac{1}{(s_0+\delta)[A-\log(s_0+\delta)]}\\
&= &\int_0^{s_0}\left(\frac{1}{[A-\log(v+\delta)](v+\delta)^2}-\frac{1}{[A-\log(v+\delta)]^2(v+\delta)^2}\right)dv\\
&\geq &\frac{1}{2}\int_0^{s_0}\frac{1}{[A-\log(v+\delta)](v+\delta)^2}dv
\end{eqnarray*}
so that
\begin{eqnarray*}
&& \frac{c_{15}\delta}{2}\int_0^{s_0}\frac{1}{[A-\log(v+\delta)](v+\delta)^2}dv\\
&\leq &c_{15}\delta\left(\frac{1}{\delta(A-\log\delta)}-\frac{1}{(s_0+\delta)[A-\log(s_0+\delta)]}\right)\\
&= &c_{15}\left(\frac{1}{A-\log\delta}-\frac{\delta}{(s_0+\delta)[A-\log(s_0+\delta)]}\right).
\end{eqnarray*}
Plugging this inequality into (\ref{E-necessary-TS-LHS-3}) yields
\begin{eqnarray}\label{E-necessary-TS-LHS-4}
& & \int_{\Sigma}f_{\delta}''\frac{|{\bf H}_{\Sigma}|^2}{|X|^{\frac{n}{2}-1}}e^{\frac{\langle X, {\bf T}\rangle}{2}}K_{\frac{n}{2}-1}\left(\frac{|{\bf T}||X-{\bf a}|}{2}\right) d\mu\\
&\geq& \frac{c_{15}}{2}\log(A-\log\delta)-\frac{c_{15}}{2}\log(A-\log(s_0+\delta))\nonumber\\
& &-c_{15}\frac{1}{A-\log\delta}+c_{15}\frac{\delta}{(s_0+\delta)[A-\log(s_0+\delta)]}.\nonumber
\end{eqnarray}

\vspace{.1in}

Combining with (\ref{E-necessary-TS-2}) and (\ref{E-RHS}), we get that
\begin{eqnarray}\label{E-necessary-TS-LHS-5}
 \frac{c_{15}}{2}\log(A-\log\delta)
&\leq & c_7+\frac{c_{15}}{2}\log(A-\log(s_0+\delta))\nonumber\\
& &+c_{15}\frac{1}{A-\log\delta}-c_{15}\frac{\delta}{(s_0+\delta)[A-\log(s_0+\delta)]}
.\nonumber
\end{eqnarray}
Now we complete the proof of Theorem \ref{Thm-1-2} by letting $\delta\to 0$.
\end{proof}

\end{document}